\documentclass{amsart}
\usepackage{amsmath,amsthm,amsfonts,amssymb,amscd,amsbsy}
\usepackage{caption,enumerate}
\usepackage{dsfont}

\usepackage{hyperref}
\usepackage{cleveref}

\hypersetup{
    pdftoolbar=true,
    pdfmenubar=true,
    pdffitwindow=false,
    pdfstartview={FitH},
    pdftitle={},
    pdfauthor={},
    pdfsubject={},
    pdfkeywords={}
    pdfnewwindow=true,
    colorlinks=true, 
    linkcolor=blue,
    citecolor=blue,
    urlcolor=black,
}

\reversemarginpar

\newcommand{\even}{{\mathrm{even}}}
\newcommand{\odd}{{\mathrm{odd}}}

\newcommand{\id}{\operatorname{id}}
\newcommand{\cl}{\mathds{C}\ell}
\newcommand{\Hom}{\operatorname{Hom}}
\newcommand{\End}{\operatorname{End}}
\newcommand{\Z}{\mathds{Z}}
\newcommand{\R}{\mathds{R}}

\newcommand{\C}{\mathds{C}}
\renewcommand{\S}{\mathds{S}}

\newcommand{\SO}{\mathsf{SO}}
\newcommand{\Spin}{\mathsf{Spin}}
\newcommand{\Sym}{\operatorname{Sym}}
\newcommand{\dd}{\mathrm d}
\newcommand{\g}{\mathrm g}

\newcommand{\diag}{\mathrm{diag}}

\newcommand{\Ric}{\operatorname{Ric}}
\newcommand{\scal}{\mathrm{scal}}
\newcommand{\II}{\mathrm I\!\mathrm I}

\allowdisplaybreaks

\newtheorem{theorem}{Theorem}[]
\newtheorem{lemma}[theorem]{Lemma}
\newtheorem{proposition}[theorem]{Proposition}
\newtheorem{corollary}[theorem]{Corollary}

\newtheorem{mainthm}{Theorem}

\newtheorem{maincor}[mainthm]{Corollary}

\theoremstyle{definition}
\newtheorem{definition}[theorem]{Definition}

\theoremstyle{remark}
\newtheorem{remark}[theorem]{Remark}
\newtheorem{example}[theorem]{Example}

\allowdisplaybreaks
\numberwithin{equation}{section}

\allowdisplaybreaks

\title[Extremality and rigidity in dimension four]{Extremality and rigidity for\\ scalar curvature in dimension four}

\subjclass{53C21, 53C23, 53C24, 53C27}

\author[R. G. Bettiol]{Renato G. Bettiol}
\address{\!\!\!\begin{tabular}{lll}
CUNY Lehman College & & CUNY Graduate Center \\
Department of Mathematics & & Department of Mathematics \\
250 Bedford Park~Blvd W & & 365 Fifth Avenue \\
Bronx, NY, 10468, USA & & New York, NY, 10016, USA
\end{tabular}
}
\email{r.bettiol@lehman.cuny.edu}

\author[M. J. Goodman]{McFeely Jackson Goodman}
\address{Colby College \newline
\indent Department of Mathematics  \newline
\indent 108 Lovejoy \newline
\indent Waterville, ME, 04901, USA}
\email{mjgoodma@colby.edu}

\allowdisplaybreaks
\numberwithin{equation}{section}
\numberwithin{theorem}{section}

\date{\today}

\begin{document}
\begin{abstract}
Following Gromov, a Riemannian manifold is called area-extremal if any modification that increases scalar curvature must decrease the area of some tangent 2-plane. We prove that large classes of compact 4-manifolds, with or without boundary, with nonnegative sectional curvature are area-extremal. We also show that all regions of positive sectional curvature on 4-manifolds are locally area-extremal. These results are obtained 
analyzing sections in the kernel of a twisted Dirac operator constructed from pairs of metrics, and using the Finsler--Thorpe trick for sectional curvature bounds in dimension 4.
\end{abstract}

\maketitle

\section{Introduction}
A lower bound $\scal\geq c>0$ on the scalar curvature of a Riemannian $n$-manifold does not constrain either its diameter or volume, provided \(n>2\). However, there are other types of upper bounds on the ``size'' of such manifolds, in terms of ``dilations of topologically significant maps'', as explained by Gromov~\cite[Sec.~4]{gromov-dozen}. 
Since shrinking areas of surfaces on manifolds with $\scal>0$ always increases scalar curvature, it is natural to ask in which cases this is the \emph{only} way of doing so.
Motivated by this question, following Gromov~\cite[Sec.~4]{gromov-dozen}, we call a metric $\g_0$ on a closed manifold $M$ \emph{area-extremal} (for scalar curvature) if, for all  metrics $\g_1$ on~$M$,
\begin{equation}\label{eq:g1-competitor-global}
\wedge^2\g_1\succeq\wedge^2\g_0 \quad\text{ and }\quad \scal(\g_1)\geq\scal(\g_0)
\end{equation}
imply that $\scal(\g_1) = \scal(\g_0)$; and \emph{area-rigid} if \eqref{eq:g1-competitor-global} implies $\g_1=\g_0$.
The first condition in \eqref{eq:g1-competitor-global} means that areas measured with \(\g_1\) are no smaller than those measured with \(\g_0\). For instance, this is the case whenever the stronger condition \(\g_1\succeq \g_0\) holds, i.e., whenever distances are no smaller under \(\g_1\) than under \(\g_0\). Thus, area-extremality can be summarized as the impossibility of increasing scalar curvature without shrinking the area of some surface (or  infinitesimal 2-plane).

Round metrics on spheres are area-rigid, by a result of Llarull \cite{llarull1}. Expanding on work of Min-Oo \cite{min-oo-scal}, Goette and Semmelmann \cite{goette-semmelmann2} proved that any metric $\g$ with positive-semidefinite curvature operator ($R\succeq0$) on a manifold with nonzero Euler characteristic is area-extremal, and area-rigid if $\tfrac{\scal}{2}\, \g\succ\Ric \succ0$. Furthermore, Goette and Semmelmann \cite{goette-semmelmann1} proved that K\"ahler metrics with \(\Ric\succeq 0\) are area-extremal, and area-rigid if \(\Ric\succ0\). However, since these previous results only apply to either spheres or manifolds with special holonomy, finding broader criteria for area-extremality/area-rigidity and describing particular classes of these metrics remain important problems \cite[Prob.~C]{gromov-dozen}.

In this paper, we prove new area-extremality and area-rigidity criteria in dimension $4$, relying on a unique characterization of sectional curvature bounds. Namely, by the so-called \emph{Finsler--Thorpe trick}, an orientable Riemannian $4$-manifold has $\sec\geq0$ if and only if there exists a function $\tau$ such that $R+\tau \, *\succeq0$, where \(R\) denotes the curvature operator, and \(*\) the Hodge star operator, each acting on 2-vectors, see \Cref{prop:FTtrick} for details. Other geometric applications of the Finsler--Thorpe trick have recently appeared in \cite{bm-iumj,bkm-siaga,bkm-geography,bk-rf}.

\subsection{Main results}
Our first result is the following extremality/rigidity criterion:

\begin{mainthm}\label{mainthm1}
	Let $(M^4,\g)$ be a closed simply-connected Riemannian manifold with $\sec\geq0$. If $\tau\colon M\to\R$ such that $R+\tau \,*\succeq0$ can be chosen nonpositive or~nonnegative, then $\g$ is area-extremal. If, in addition, $\tfrac{\scal}{2}\, \g\succ\Ric \succ0$, then $\g$ is area-rigid.
\end{mainthm}

As a consequence of \cite[Thm.~D]{bm-iumj}, the $4$-manifolds $(M^4,\g)$ to which \Cref{mainthm1} applies either have
\emph{definite} intersection form, or are isometric to $\S^2\times \S^2$ endowed with a product metric.
By classical work of Donaldson and Freedman, the former are homeomorphic to \(\S^4\) or to a connected sum $\C P^2\#\dots\#\C P^2$ of finitely many copies of $\C P^2$. Conjecturally, such manifolds only admit metrics with
$\sec>0$ if at most one summand is used (i.e., if $M^4$ is either $\S^4$ or $\C P^2$), and with $\sec\geq0$ if at most two summands are used (i.e., if $M^4$ is either $\S^4$, $\C P^2$, or $\C P^2\#\C P^2$). 

Previously known area-extremal metrics on simply-connected $4$-manifolds $M$~are:
\begin{itemize}
	\item metrics with $R\succ0$, hence $M$ is diffeomorphic to $\S^4$, see e.g.~\cite[Thm.~1.10]{wilking-survey};
	\item metrics with $R\succeq0$ (but $R\not\succ0$), hence $M$ is isometric to either a product metric on $\S^2\times \S^2$ or a K\"ahler metric on $\C P^2$, see e.g.~\cite[Thm.~1.13]{wilking-survey};
	\item K\"ahler metrics on $\C P^2\# \overline{\C P^2}$ with \(\Ric\succ 0\), which exist by Yau's solution \cite{yau-calabiconj} to the Calabi conjecture, and are area-rigid by \cite{goette-semmelmann1}. 
\end{itemize}
Thus, using \Cref{mainthm1}, we obtain new examples of area-rigid metrics:

\begin{maincor}\label{maincor:CP2CP2}
	The following hold:
	\begin{enumerate}[\rm (i)]
		\item On $\C P^2$, metrics in a neighborhood of the Fubini--Study metric are area-rigid;
		\item On $\C P^2\# \C P^2$, Cheeger metrics with vanishing neck length are area-rigid.
	\end{enumerate}
\end{maincor}

To our knowledge, \Cref{maincor:CP2CP2} (i) gives the first example of an \emph{open set} of area-rigid metrics on a closed manifold other than the sphere. 
Moreover, note that
both (i) and (ii) yield the existence of area-rigid metrics with \emph{generic holonomy}. Recall that \emph{Cheeger metrics} on the connected sum of two compact rank one symmetric spaces are metrics with $\sec\geq0$ obtained gluing complements of disks along a ``neck'' region isometric to a round cylinder $\S^n\times [0,\ell]$ of arbitrary length $\ell\geq0$, see \cite{cheeger}. 
In particular, the neck region does not have $\Ric\succ0$. For this reason, $\ell=0$ is necessary in \Cref{maincor:CP2CP2} (ii), since then $\Ric\succ0$ on the (dense) complement of the neck hypersurface. Meanwhile, Cheeger metrics on $\C P^2\# \C P^2$ with $\ell>0$ are area-extremal but not area-rigid, see \Cref{thm:cheeger-metrics} and \Cref{rem:noneck}.
We note that $\C P^2\# \C P^2$ admits no metrics with $R\succeq0$, nor complex structures, and thus no previous methods can be used to identify area-extremal metrics on this manifold.

As a consequence of \Cref{maincor:CP2CP2} and previous examples,
there are area-extremal metrics with $\sec\geq0$ on all closed simply-connected $4$-manifolds currently known to admit metrics with $\sec\geq0$, namely, $\S^4$, $\C P^2$, $\S^2\times \S^2$, and $\C P^2\# \pm\C P^2$, which are conjectured to be all of them.
Furthermore, metrics in a neighborhood of the round metric on \(\S^4\) have $R\succ0$ and are thus area-rigid by \cite{goette-semmelmann2}, so we conclude
there is an open set of area-rigid metrics with $\sec>0$ on each of the two simply-connected closed $4$-manifolds known (and conjectured to be all) to admit metrics with \(\sec>0\).

\smallskip
There is a natural extension of the above notions of area-extremality and area-rigidity to Riemannian manifolds \emph{with boundary}. 
Given such a manifold $(M,\g)$, we denote by $\II_{\partial M}$ the second fundamental form of $\partial M$ with respect to the inward unit normal, and by $H(\g)=\operatorname{tr}\II_{\partial M}$ the mean curvature of $\partial M$.
A metric $\g_0$ on $M$ is \emph{area-extremal} (for scalar curvature) if all metrics $\g_1$ satisfying \eqref{eq:g1-competitor-global} on $M$ with
\begin{equation}\label{eq:g1-competitor-local}
\g_1|_{\partial M}=\g_0|_{\partial M} \quad\text{ and }\quad H(\g_1)\geq H(\g_0)
\end{equation}
must have $\scal(\g_1) = \scal(\g_0)$ as well as $H(\g_1) = H(\g_0)$; and \emph{area-rigid} if \eqref{eq:g1-competitor-global} and \eqref{eq:g1-competitor-local} imply $\g_1=\g_0$. It is striking that area-extremality, in the above sense for $4$-manifolds with boundary, holds locally around any point with $\sec>0$, as follows:

\begin{mainthm}\label{mainthm-local1}
	If \((X^4,\g)\) is a Riemannian \(4\)-manifold, then $\g$ is area-extremal on sufficiently small convex neighborhoods $M$ of any point in $X$ at which $\g$ has $\sec>0$.
\end{mainthm}

In particular, \Cref{mainthm-local1} implies that any metric deformation supported in a sufficiently small convex subset $M$ of a $4$-manifold with $\sec>0$ either preserves both $\scal$ and $H$, or else decreases the area of some surface in $M$, or $\scal$ somewhere on $M$, or $H$ somewhere on $\partial M$.  On the other hand, recall that for generic $\g$, \emph{any} sufficiently small deformation of the function $\scal|_M$ is realized as the (restriction to $M$ of the) scalar curvature of a metric near $\g$, by a result of Corvino~\cite{corvino}.

\Cref{mainthm-local1} follows from an extension of \Cref{mainthm1} to $4$-manifolds $(M^4,\g)$ with convex boundary, i.e., with $\II_{\partial M}\succeq0$. (Note that a geodesic ball of sufficiently small radius in any Riemannian manifold has convex boundary.)

\begin{mainthm}\label{mainthm-local2}
	Let $(M^4,\g)$ be an orientable compact Riemannian $4$-manifold with  $\sec\geq0$ and $\II_{\partial M}\succeq0$. If $\tau\colon M\to \R$ such that $R+\tau\,*\succeq0$ can be chosen nonpositive or nonnegative, then $\g$ is area-extremal. If, in addition, $\frac{\scal}{2}\,\g\succ\Ric\succ0$, then $\g$ is area-rigid.
\end{mainthm}

Similar criteria for manifolds with boundary have been recently proven in all even dimensions under the more restrictive assumption $R\succeq0$.
Lott \cite{lott-spin} showed that metrics with $R\succeq0$ and $\II_{\partial M}\succeq0$ on compact manifolds with boundary and nonzero Euler characteristic are area-extremal (in a more general sense; namely, relaxing the first condition in \eqref{eq:g1-competitor-local} to \(\g_1|_{\partial M}\succeq \g_0|_{\partial M}\)).
Cecchini and Zeidler \cite{CZ} obtained area-extremality results for certain warped product metrics on \(M\times [-1,1],\) where \(M\) has $R\succeq0$ and nonvanishing Euler characteristic; those results have recently been extended in \cite{BBHW}. 
It is noteworthy that \Cref{mainthm-local2} does not require \(M\) to be simply-connected, nor to have nonzero Euler characteristic. Indeed, by the Soul Theorem, any $(M^4,\g)$ as in \Cref{mainthm-local2} is a disk bundle over a totally geodesic closed submanifold, whose topology is constrained by the fact it has $\sec\geq0$ and dimension $\leq3$. In turn, this has topological implications on $M$ which are sufficient to apply our methods and prove area-extremality.

\Cref{mainthm-local2} yields new examples of area-extremal and area-rigid metrics; e.g., the metrics on the complement of a ball in $\C P^2$ used in the construction of Cheeger metrics on \(\C P^2\# \C P^2\), described above. 
No other previous criteria apply to this manifold, since it is diffeomorphic to the normal disk bundle of \(\C P^1\subset\C P^2\), which is known not to admit metrics with $R\succeq0$ and $\II_{\partial M}\succeq0$, see~\cite{Noronha}.
Examples with vanishing Euler characteristic are provided by standard metrics with $R\succeq0$ on products of spheres and disks, see \Cref{ex:examples-local}.
With different boundary conditions, such examples are also addressed in the related result~\cite[Thm.~1.3]{lott-spin}. Finally, note that the round hemisphere $\S^4_+$ is area-rigid as a consequence of \Cref{mainthm-local2}, or \cite[Cor.~1.2]{lott-spin}, so the counterexamples to the Min-Oo conjecture constructed in \cite{bmn} must shrink areas somewhere on $\S^4_+$, a fact that was shown in~\cite{miao-tam}. 

\smallskip
Both Theorems \ref{mainthm1} and \ref{mainthm-local2} follow from more general results (\Cref{thm,local}), in which we prove area-extremality/area-rigidity in a broader sense, also discussed in \cite{llarull1, goette-semmelmann1,gromov-dozen,lott-spin}. Namely, metrics can be compared similarly to \eqref{eq:g1-competitor-global} and \eqref{eq:g1-competitor-local} but using maps other than the identity, including maps between different manifolds, which Gromov describes as allowing for competitors with ``topological modifications''. Indeed, Theorems \ref{mainthm1} and \ref{mainthm-local2} are simplified versions of such statements on comparisons with self-maps of nonzero degree, see \Cref{cor:self-maps,cor:boundary-self}.

\subsection{Outline of proofs}
For the reader's convenience, we briefly describe a general framework to prove area-extremality/area-rigidity based on spin geometry, which is used in the above results.
Let $(M,\g_0)$ and $(N,\g_1)$ be oriented Riemannian manifolds and \(f\colon N\to M\) be a spin map; e.g., one may take $N=M$ and $f=\id$; see \Cref{4closed} for details. Consider the Dirac operator \(D(\g_0,\g_1)\) on spinors over \(N\) twisted with the pullback bundle via $f$ of the spinor bundle over \(M.\) 

The Bochner--Lichnerowicz--Weitzenb\"ock formula for \(D(\g_0,\g_1)\) is given by 
\[D(\g_1,\g_0)^2=\nabla^*\nabla+\tfrac14\scal(\g_1)+\mathcal{R}(R,\dd f),\]
where \(\mathcal R(R,\dd f)\) is a bundle endomorphism that depends only on the curvature operator \(R\) of \((M,\g_0)\) and  \(\dd f\colon TN\to TM\). Algebraic considerations show that
\begin{equation*}
\mathcal{R}(R,\dd f)=\mathcal{T}(R,\dd f)-\tfrac14 \operatorname{tr}(R)\circ f-\tfrac14\operatorname{tr}\!\big( F^*\circ R\circ F\big)\circ f,
\end{equation*}
where $F\colon\wedge^2 TN\to\wedge^2 TM$ is the map $F(v\wedge w)=\dd f(v)\wedge \dd f(w)$, and $\mathcal{T}(R,\dd f)\succeq0$ whenever $R\succeq0$, see \Cref{rt,nneg}. For instance, if \(f=\id\) and \(\g_1=\g_0\), then \(\mathcal{T}(R,\dd f)\) is the curvature term in the Weitzenb\"ock formula for the Hodge Laplacian on forms. Further algebraic considerations (see \Cref{fscal}) show that if \(\wedge^2\g_1\succeq f^*\wedge^2\g_0\) and $R$ has \(\sec\geq0\), then
\(
\operatorname{tr}\!\big(F^*\circ R\circ F\big)\leq\operatorname{tr}(R)=\tfrac12\scal(\g_0)\).

We thus search for conditions on \(R\) such that \(\sec\geq0\) and \(\mathcal T(R,\dd f)\succeq0\) whenever \(\wedge^2\g_1\succeq f^*\wedge^2\g_0\), since then
\[D(\g_0,\g_1)^2\succeq \nabla^*\nabla+\tfrac14\big(\scal(\g_1)-\scal(\g_0)\circ f\big)\]
and so \(\scal(\g_1)\geq\scal(\g_0)\circ f\) forces \(\scal(\g_1)=\scal(\g_0)\circ f\) if the underlying topologies imply, by way of the Atiyah--Singer Index Theorem, that \(\ker D(\g_1,\g_0)\neq\{0\}\).

In \Cref{pi}, we show that the curvature assumption in \Cref{mainthm1} ensures that \(\mathcal T(R,\dd f)\succeq0\), up to restricting this endomorphism to an appropriate subbundle. In \Cref{indextheory}, we give topological conditions on $M$ and $N$ sufficient to have a nontrivial section \(\xi\) of that subbundle with \(D(\g_0,\g_1)\xi=0\).
In the earlier area-extremality/area-rigidity works mentioned above, this topological condition is the nonvanishing of the Euler characteristic, and the corresponding grading is used on the tensor product of spinor bundles.
However, the specific subbundle stemming from computations in \Cref{pi} requires us to use a novel method, combining the so-called ``Euler characteristic" and ``signature" gradings of the tensor product of spinor bundles, an idea reminiscent of \cite[Rmk.~2.3]{goette-semmelmann2}. We then use \(\xi\) to prove our most general area-extremality and area-rigidity result for closed manifolds, \Cref{thm}, which in turn implies \Cref{mainthm1}. In \Cref{examples}, we describe metrics satisfying the required curvature assumption, proving \Cref{maincor:CP2CP2}.  

Following \cite{lott-spin}, the above method is extended to manifolds with boundary in \Cref{4boundary}, proving \Cref{local} along with two corollaries, which imply Theorems \ref{mainthm-local1} and \ref{mainthm-local2}. 
The convexity assumption $\II_{\partial M}\succeq0$ is used to ensure that \(\nabla^*\nabla\succeq0\), while the existence of a nontrivial section $\xi$ of the appropriate subbundle involves an application of the Atiyah--Patodi--Singer Index Theorem for manifolds with boundary.

\subsection*{Acknowledgements} It is our great pleasure to thank Anusha Krishnan, Claude LeBrun, and John Lott for valuable conversations regarding Grove--Ziller metrics, K\"ahler metrics, and Index Theory on manifolds with boundary, respectively. We also thank the referee for the careful reading of our paper and insightful suggestions.

The first-named author is supported by the National Science Foundation, through grant DMS-1904342 and CAREER grant DMS-2142575. The second-named author is supported by the National Science Foundation, through grant DMS-2001985.

\section{Preliminaries}\label{prelim}

In this section, we fix conventions, definitions, and notations, and recall basic facts from linear algebra and spin geometry, closely following \cite{lm-book}. Throughout, $V$ and $W$ denote (finite-dimensional) oriented real inner product spaces.

\subsection{Linear algebra}\label{subsec:linalgebra}
Given a linear map $l\colon W\to V$, its \emph{adjoint} $l^*\colon V\to W$ is the linear map such that $\langle l(w),v\rangle_V = \langle w,l^*(v)\rangle_W$ for all $v\in V$ and $w\in W$.
The space of linear maps $l\colon W\to V$ is denoted $\Hom(W,V)$, and, if $V=W$, we write $\End(V)=\Hom(V,V)$. We identify $\Hom(W,V)=W\otimes V$ by means of $(w\otimes v)(\cdot)=\langle w,\cdot\rangle v$. The subspaces of \emph{symmetric} and \emph{skewsymmetric} endomorphisms of $V$, i.e., $l\in\End(V)$ such that $l^*=l$ and $l^*=-l$, are denoted $\Sym^2(V)\subset \End(V)$ and $\wedge^2 V\subset \End(V)$, respectively, and $\End(V)=\Sym^2(V)\oplus\wedge^2 V$. The special orthogonal group of $V$, i.e., the group of linear isometries of $V$ is denoted $\SO(V)\subset \Sym^2(V)$.

All of $\End(V)$, $\Sym^2(V)$, $\wedge^2 V$, $\Hom(W,V)$, $\Hom(\wedge^2 W,\wedge^2 V)$, etc., are endowed with the compatible inner products determined by those in $V$ and $W$. 
For example, if a linear map $L\colon \wedge^2 W\to\wedge^2 V$ is of the form $L=\wedge^2 l$ for some $l\colon W\to V$,~i.e.,
\begin{equation}\label{eq:L}
\phantom{, \quad\text{ for all } w_1,w_2\in W}
    L(w_1\wedge w_2)=l(w_1)\wedge l(w_2), \quad\text{ for all } w_1,w_2\in W,
\end{equation}
then its adjoint \(L^*\colon \wedge^2V\to \wedge^2W\) is given by $L^*=\wedge^2 (l^*)$. 

We will make repeated use of the following elementary fact from Linear Algebra:

\begin{lemma}[Singular Value Decomposition]\label{lemma:diag}
    Given any linear map $l\colon W\to V$ between real inner product spaces of the same dimension, there exist orthonormal bases $\{w_i\}$ of $W$ and $\{v_i\}$ of $V$, and real numbers $\lambda_i\geq0$, such that $l(w_i)=\lambda_i\,v_i$.
\end{lemma}

\begin{proof}
Let $A\in\Sym^2(W)$ be the linear map such that $\langle A x,y\rangle_W=\langle l(x),l(y)\rangle_V$ for all $x,y\in W$. Clearly, $A$ is symmetric and positive-semidefinite, so it can be diagonalized by an orthonormal basis $\{w_i\}$ of $W$, on which $A=\diag(\lambda_i^2)$ for some $\lambda_i\geq 0$. By the above, $\{l(w_i)\}$ are pairwise orthogonal vectors in $V$ that span the image $l(W)$.  The unit-length vectors corresponding to the nonzero $l(w_i)$ form a set that can be extended to an orthonormal basis $\{v_i\}$ of $V$, and, by construction, $l(w_i)=\lambda_i\,v_i$.
\end{proof}

We say that a linear map $l\colon W\to V$ is \emph{nonincreasing} if $\|l(w)\|\leq \|w\|$ for all $w\in W$, or, equivalently, if all $\lambda_i\geq0$ arising from \Cref{lemma:diag} satisfy $\lambda_i\leq 1$. Note that if $l\colon W\to V$ is nonincreasing, then so is $L=\wedge^2l\colon \wedge^2 W\to\wedge^2 V$ given in \eqref{eq:L}.

\subsection{Algebraic curvature operators}
For convenience, we treat endomorphisms $R\in\Sym^2(\wedge^2 V)$ both as a symmetric linear maps $R\colon \wedge^2 V\to \wedge^2 V$ and as linear maps
$V\times V\ni (x,y)\mapsto R_{x,y}\in \wedge^2 V\subset \End(V)$ such that $R_{x,y}=-R_{y,x}$, where
\begin{equation}\label{eq:curvop-sign}
	\phantom{\quad\text{for all}x,y,z,w\in V.}
	\langle R_{x,y}(z),w\rangle=\langle R(x\wedge y),w\wedge z\rangle, \quad\text{ for all }x,y,z,w\in V.
\end{equation}
The linear subspace of $\Sym^2(\wedge^2 V)$ formed by those $R$ that satisfy the \emph{first Bianchi identity} $R_{x,y}(z)+R_{y,z}(x)+R_{z,x}(y)=0$ is denoted $\Sym^2_b(\wedge^2 V)$, and its elements are called \emph{algebraic curvature operators}. These are pointwise models at each $V=T_pM$ for the curvature tensor/operator $R$ of a Riemannian manifold $(M,\g)$. Accordingly, the \emph{sectional curvature} of a $2$-plane $\sigma=x\wedge y\in\wedge^2 V$ with respect to $R$ is
\begin{equation*}
\sec_R(\sigma)=\langle R_{x,y}(y),x\rangle=\langle R(\sigma),\sigma\rangle ,
\end{equation*}
while $\Ric_{R}(x,y)$ is the trace of the endomorphism $z\mapsto R_{z, x}(y)$, and $\scal_{R}=2\operatorname{tr} R$. As usual, by $\sec_R\geq0$ and $\Ric_R\succeq0$ we mean $\sec_R(\sigma)\geq0$ for all $2$-planes $\sigma\subset V$ and $\Ric_R(x,x)\geq0$ for all $x\in V$, respectively; similarly for $\sec_R>0$ and $\Ric_R\succ0$.

The orthogonal complement of $\Sym^2_b(\wedge^2 V)$ in $\Sym^2(\wedge^2 V)$ can be identified with $\wedge^4 V$, where $\omega\in\wedge^4 V\subset\Sym^2(\wedge^2 V)$ is given by $\langle \omega(\alpha),\beta \rangle=\langle \omega,\alpha\wedge\beta\rangle$. In particular, if $\dim V=4$, this is a one-dimensional space spanned by the Hodge star operator $*\colon\wedge^2 V\to\wedge^2 V$.
Moreover, $\sigma\in\wedge^2V$ satisfies 
$\sigma\wedge\sigma=0$ if and only if $\langle *\sigma,\sigma\rangle=0$, i.e., the quadric defined by $*$ in $\wedge^2 V$ is precisely the Pl\"ucker embedding of the oriented Grassmannian of $2$-planes $\operatorname{Gr}_2^+(V)\subset\wedge^2V$.
As shown by Finsler~\cite{finsler}, a quadratic form $\langle R(\sigma),\sigma\rangle$ is nonnegative when restricted to the quadric $\langle *\sigma,\sigma\rangle=0$ if and only if some linear combination of $R$ and $*$ is positive-semidefinite, yielding:

\begin{proposition}[Finsler--Thorpe trick]\label{prop:FTtrick}
Let $R\in \Sym^2_{b}(\wedge^2 V)$ be an algebraic curvature operator on $V$, with $\dim V=4$. Then $\sec_R\geq0$, respectively $\sec_R>0$, if and only if there exists $\tau\in\R$ such that $R+\tau\, *\succeq0$, respectively $R+\tau\, *\succ0$.
\end{proposition}

\begin{remark}
The above has been referred to as \emph{Thorpe's trick}, as it was rediscovered by Thorpe~\cite{Thorpe72}, see~\cite{bkm-siaga} for details. 
In the mathematical optimization and control literature, this fact is known as \emph{$S$-lemma}, or \emph{$S$-procedure}, see \cite{slemma-survey}.
\end{remark}

It is an easy consequence of convexity that the set of $\tau\in\R$ such that $R+\tau\,*\succeq0$ for a fixed $R\in \Sym^2_{b}(\wedge^2 V)$ with $\sec_R\geq0$, as in \Cref{prop:FTtrick}, is a closed interval $[\tau_{\mathrm{min}},\tau_{\mathrm{max}}]$, which degenerates to a single point  (i.e., $\tau_{\mathrm{min}}=\tau_{\mathrm{max}}$) if and only if $R$ has $\sec_R\geq0$ but does not have $\sec_R>0$, see \cite[Prop.~3.1]{bkm-siaga}.

\subsection{Clifford Algebra and Spinors}\label{subsec:CAandS}
If $\dim_\R V=2n$, then the complex Clifford algebra \(\cl(V)\) associated to \(V\) has a unique irreducible complex representation \(S(V)\), which is a complex vector space of dimension $2^{n}$ of so-called \emph{(Dirac) spinors}. We endow $S(V)$ with a Hermitian inner product for which the action of unit vectors in \(\cl(V)\) is isometric. The complex volume element \(\omega_\C\in\cl(V)\), which is given by \(\omega_{\C}=(\sqrt{-1})^n e_1e_2\dots e_{2n}\) for any  orthonormal basis \(\{e_1,\dots,e_{2n}\}\) of \(V\), satisfies \(\omega_\C^2=1\). Thus, it induces splittings of $\cl(V)$ and $S(V)$ as orthogonal direct sums of eigenspaces of \(\omega_\C\) with eigenvalue \(\pm1\), denoted
\begin{equation}\label{eq:clpm-spm}
    \cl(V)=\cl^+(V)\oplus \cl^-(V) \quad\text{ and }\quad S(V)=S^+(V)\oplus S^-(V),
\end{equation}
respectively. We write $S:=S(V)$ and $S^\pm:=S^\pm(V)$ to simplify notation, when the inner product space $V$ in question is clear from the context.

The homomorphism \(\cl(V)\to \End(S)\) defining the representation $S$ is an isomorphism; and, just as in the real case discussed above, the Hermitian inner product $\langle,\rangle$ on \(S\) 
allows us to identify \(\End(S) =S\otimes S\) via \((\phi\otimes \psi)(\cdot)=\langle \phi,\cdot \rangle \, \psi\), for all \(\phi,\psi\in S\). The composition of these isomorphisms is an isomorphism \(\cl(V)\cong S\otimes S\) which is \(\cl(V)\)-equivariant with respect to left multiplication on \(\cl(V)\) and multiplication on the second factor of \(S\otimes S\). Thus, in light of \eqref{eq:clpm-spm}, it restricts to isomorphisms 
\begin{equation}\label{eq:cl+cl-}
    \cl^+(V)\cong S\otimes S^+\quad \text{ and }\quad\cl^-(V)\cong S\otimes S^-.
\end{equation}

As a vector space, the Clifford algebra \(\cl(V)\) is isomorphic to the complexified exterior algebra \(\wedge^*_\C V=\bigoplus_p \wedge^p V\otimes_\R \C\), via the linear map given on orthonormal basis elements by \(e_{i_1}\dots e_{i_p}\mapsto e_{i_1}\wedge\dots\wedge e_{i_p}\). This is a \emph{$\Z_2$-graded} isomorphism: the natural splitting \(\cl(V)=\cl^0(V)\oplus \cl^1(V)\) arising from the $\Z_2$-grading of $\cl(V)$ is mapped to the splitting $\wedge^*_\C V=\wedge^\even_\C V\oplus \wedge^\odd_\C V$ into exterior powers of even and odd degrees.
The action of \(\cl^0(V)\cong \wedge^\even_\C V\) on $S$ preserves \(S^\pm\) while that of \(\cl^1(V)\cong \wedge^\odd_\C V\) interchanges these subspaces.
Since the $\Z_2$-graded isomorphism \(\wedge^*_\C V\to \cl(V)\) conjugates
the duality isomorphism \(\wedge^pV\to \wedge^{2n-p} V\) given by $(\sqrt{-1})^{p(p-1)+n}\,*$, where \(*\) is the Hodge star operator, and left multiplication by \(\omega_\C\) in \(\cl(V)\),
it follows that
\begin{align*}
\wedge_\C^\even V&=\wedge_\C^{+,\even} V \oplus \wedge_\C^{-,\even} V, &&  \wedge_\C^\odd V=\wedge_\C^{+,\odd} V\oplus \wedge_\C^{-,\odd} V, \\ 
\cl^0(V)&=\cl^{+,0}(V)  \oplus \cl^{-,0}(V), && \cl^1(V)=\cl^{+,1}(V)\oplus \cl^{-,1}(V),
\end{align*}
where vertically aligned spaces are isomorphic, i.e.,
\begin{equation}\label{decompositions}
\begin{aligned}
\wedge_\C^{\pm,\even}V &\cong \cl^{\pm,0}(V)\cong S^\pm \otimes S^\pm,\\
\wedge_\C^{\pm,\odd}V &\cong \cl^{\pm ,1}(V)\cong S^\mp \otimes S^\pm ,
\end{aligned}
\end{equation}
and this notation is compatible with \eqref{eq:cl+cl-}, i.e., $\cl^\pm(V)=\cl^{\pm,0}(V)\oplus\cl^{\pm,1}(V)$.

\subsection{Spin group}
The {nontrivial} double cover of the special orthogonal group $\SO(V)$ is the \emph{spin group} $\Spin(V)$, and it can be realized as $\Spin(V)\subset \cl^0(V)$, see \cite[Chap.~I]{lm-book} for details. Thus, its Lie algebra is isomorphic to the Lie algebra $\mathfrak{so}(V)$ of $\SO(V)$, which is identified with $\wedge^2 V$ as usual, i.e., $(x\wedge y)(\cdot)=\langle x,\cdot \rangle y-\langle y,\cdot\rangle x$ corresponds to an infinitesimal rotation in the $2$-plane of $V$ spanned by $x$ and $y$.

The inverse of the Lie algebra isomorphism $\Xi_0$ induced by the double cover \(\cl^0(V)\supset\Spin(V)\to \SO(V)\) is the map $\Xi^{-1}_0 \colon \wedge^2V\cong \mathfrak{so}(V)\to\mathfrak{spin}(V)\subset \cl^0(V)$ given on $x\wedge y$, where $x,y\in V$ are orthogonal vectors, by
\begin{equation}\label{coverderiv}
\Xi^{-1}_0(x\wedge y)=\tfrac12 xy,
\end{equation}
cf.~\cite[Prop.~I.6.2]{lm-book}. 
Note that \(\Xi_0^{-1}\) differs by a factor of \(\tfrac12\) from the restriction to $\wedge^2 V$ of the isomorphism \(\wedge^*_\C V\to \cl(V)\) mentioned above, for which $x\wedge y\mapsto xy$.

\subsection{Spinor bundles and Dirac operators}
Let $(M^{2n},\g)$ be an oriented Riemannian manifold of dimension $2n$, and denote by \(\nabla^\text{LC}\) its Levi-Civita connection on \(TM\). Applying the above constructions pointwise, i.e., to each tangent space $V=T_pM$, $p\in M$, we obtain the \emph{spinor bundle} $S(TM)$ over $M$ and analogous isomorphisms and splittings compatible with the natural connections induced by \(\nabla^\text{LC}\) on each of these bundles. In particular, we note that \(\omega_\C\) is parallel.

Once again, to simplify notation, we write $S:=S(TM)$ and $S^\pm:=S^\pm(TM)$ if the Riemannian manifold $(M,\g)$ is clear from the context, as well as $S_\g$ and~$S^\pm_\g$, respectively, to indicate the Riemannian metric $\g$ being used, when necessary. The connection on \(S_\g\) induced by \(\nabla^\text{LC}\) is obtained applying the map \eqref{coverderiv} to the connection forms.  Namely, given a local orthonormal frame \(\{e_1,\dots,e_{2n}\}\) of $TM$ we can choose a local frame for \(S\) such that the connection \(\nabla^S\) on \(S\) is given by
\begin{equation*}
\nabla^S_{v}=\dd+\sum_{i<j}\g\big(\nabla^\text{LC}_{v}e_i,e_j\big)\frac{e_ie_j}{2}. 
\end{equation*}
The curvature tensor \(R^S\colon TM\times TM\to\wedge^2 S\subset \End(S)\) of \(\nabla^S\) is given by
\begin{equation}\label{spinorcurv}
R^S_{x,y}=\Xi_0^{-1}\circ R_{x,y}=\sum_{i<j} \g\big( R_{x,y}(e_i),e_j \big)\frac{e_i e_j}{2},
\end{equation}
where \(R\colon TM\times TM\to\wedge^2TM\subset \End(TM)\) is the curvature tensor of \(\nabla^\text{LC}\). 

The \emph{Dirac operator} is the first-order differential operator on sections of $S$ given~by
\begin{equation*}
D(\phi)=\sum_{i=1}^{2n}e_i\nabla^S_{e_i}\phi.    
\end{equation*}
More generally, if \(E\) is a complex vector bundle over \(M\), with a connection \(\nabla^E\),
we can consider the \emph{spinor bundle twisted by} $E$, which is the bundle \(S\otimes E\) endowed with the tensor product connection \(\nabla\) and Clifford multiplication on the \(S\) factor, which makes it a Clifford bundle with a natural \emph{twisted Dirac operator} \(D_E\). Similarly to the above, \(D_E\) acts on a decomposable local section \(\phi\otimes \varepsilon\) of \(S\otimes E\) by 
\begin{equation}\label{eq:twisted-Dirac}
D_E(\phi\otimes \varepsilon)=\sum_{i=1}^{2n}(e_i\nabla^S_{e_i}\phi)\otimes \varepsilon+(e_i \phi)\otimes \nabla^E_{e_i}\varepsilon.    
\end{equation}
The \emph{Bochner--Lichnerowicz--Weitzenb\"ock formula}, cf.~\cite[Thm.~II.8.17]{lm-book}, relates $D_E^2$ and the connection Laplacian $\nabla^*\nabla$ acting on sections of $S\otimes E$ as follows
\begin{equation}\label{twistlich}
D^2_E=\nabla^*\nabla+\tfrac{1}{4}\scal(\g)+\sum_{i<j}e_ie_j\otimes R^E_{e_i,e_j},
\end{equation}
where \(R^E\colon TM \times TM\to \wedge^2 E\subset \End(E)\) is the curvature tensor of \(\nabla^E\), and $\scal(\g)$ is the scalar curvature of $(M,\g)$. For example, if $E$ is the trivial bundle, one recovers the well-known formula $D^2=\nabla^*\nabla+\tfrac14\scal(\g)$ for sections of $S$; while if $E=S$, the twisted Dirac operator $D_S$ on $S\otimes S$ is conjugate to \(\dd+\dd^*\) acting on \(\wedge_\C^* TM^*\cong \cl(TM) \cong S\otimes S\), via the isomorphisms above.
 
\section{Pointwise Inequalities}\label{pi}

In this section, we analyze algebraic properties and provide estimates for two types of curvature terms: the last term $\mathcal R$ in the Bochner--Lichnerowicz--Weitzenb\"ock formula \eqref{twistlich} if $E=f^*(S(TM))$ is the pullback by $f\colon N\to M$ of the spinor bundle of $M$, and (a modification of) the curvature term $\mathcal T$ in the Weitzenb\"ock formula for the Hodge Laplacian on differential forms on $M$. This is done pointwise, so we work with oriented real inner product spaces \(V\) and \(W\) of the same dimension, a linear map $l\colon W\to V$ which encodes $\dd f$, and the induced map $L=\wedge^2 l$ as in \eqref{eq:L}.

\begin{definition}\label{def:RT}
For every \(R\in\Sym^2(\wedge^2 V)\) and $L\in\Hom(\wedge^2 W,\wedge^2 V)$, 
we define two elements in the space of endomorphisms $\End\!\big(S(W)\otimes S(V)\big)$ by means of
\begin{align*}
\mathcal{R}(R,L)&:=-2\sum_{i}\beta_i\otimes R(L(\beta_i)),\\
\mathcal{T}(R,L)&:=-\sum_i\big(L^*(\alpha_i)\otimes1+1\otimes \alpha_i\big)\circ\big(L^*(R(\alpha_i))\otimes 1+1\otimes R(\alpha_i)\big),
\end{align*}
where \(\{\alpha_i\}\) and \(\{\beta_i\}\) are orthonormal bases of \(\wedge^2 V\) and \(\wedge^2 W\), respectively, and
\begin{equation}\label{eq:inclusions}
\wedge^2 V \subset \End\!\big(S(V)\big) \quad \text{ and } \quad \wedge^2 W \subset \End\!\big(S(W)\big)   
\end{equation}
via the respective actions of $\wedge^2 V$ and $\wedge^2 W$ on $S(V)$ and $S(W)$, determined by the map \(\Xi_0^{-1}\) in \eqref{coverderiv}. Moreover, we canonically identify $\End\!\big(S(W)\big) \otimes \End\!\big(S(V)\big)$ and $\End\!\big(S(W)\otimes S(V)\big)$, and $\circ$ in the definition of $\mathcal T(R,L)$ is composition in the latter.
\end{definition}

The endomorphisms $\mathcal{R}(R,L)$ and $\mathcal{T}(R,L)$ of $S(W)\otimes S(V)$ do not depend on the choices of \(\{\alpha_i\}\) and \(\{\beta_i\}\). Indeed, identifying $\wedge^2 W\otimes \wedge^2 V = \Hom(\wedge^2 W,\wedge^2 V)$, and considering $R\in\Sym^2(\wedge^2 V)\subset\End(\wedge^2 V)=\wedge^2 V\otimes\wedge^2 V$, we have that
\begin{equation*}
 \mathcal R(R,L) = -2\, R\circ L\quad \text{ and }\quad \mathcal T(R,L)=-c((T_L\otimes T_L)(R)),   
\end{equation*}
 where $T_L\colon \wedge^2 V\to\End(S(W)) \otimes \End(S(V))$ is given by $T_L(\alpha)=(L^*(\alpha)\otimes 1+1\otimes \alpha)$, and $c$ is the composition $c(A\otimes B)=A\circ B$ for all $A,B\in \End\!\big(S(W)\otimes S(V)\big)$. Clearly, the maps $R\mapsto \mathcal{R}(R,L)$, $L\mapsto \mathcal R(R,L)$, and $R\mapsto \mathcal{T}(R,L)$ are linear.

As elements of $\cl(W)\otimes\cl(V)\cong \End(S(W))\otimes \End(S(V))$, both $\mathcal R(R,L)$ and $\mathcal T(R,L)$ belong to $\cl^0(W)\otimes\cl^0(V)$, and hence, as endomorphisms, they restrict to endomorphisms of $S^+(W)\otimes S^+(V)$ and of $S^-(W)\otimes S^-(V)$.

\begin{lemma}\label{rt}
For all algebraic curvature operators \(R\in\Sym_b^2(\wedge^2V)\), we have
\begin{equation*}
\mathcal{R}(R,L)=\mathcal{T}(R,L)-\tfrac14\operatorname{tr}(L^*\circ R\circ L)-\tfrac18\scal_{R}.
\end{equation*}
\end{lemma}

\begin{proof}
By \Cref{lemma:diag}, we may choose orthonormal bases \(\{\alpha_i\}\)  of $\wedge^2V$ and \(\{\beta_i\}\) of $\wedge^2 W$ such that \(L(\beta_i)=\lambda_i\alpha_i\) for some \(\lambda_i\geq0\). Note that $L^*(\alpha_i)=\lambda_i\beta_i$. Since \(R\) is symmetric, we may write $R(\alpha_i)=\sum_j R_{ij}\alpha_j$ with $R_{ij}=R_{ji}$, and hence
\begin{align*}
    \mathcal{R}(R,L)&=-\sum_i \beta_i\otimes R(\lambda_i\alpha_i)-\sum_i \beta_i\otimes \lambda_i \sum_j R_{ij}\alpha_j \\
    &=-\sum_i \lambda_i\beta_i\otimes R(\alpha_i)-\sum_j \sum_i  R_{ji} \lambda_i \beta_i\otimes \alpha_j \\
    &=-\sum_i L^*(\alpha_i)\otimes R(\alpha_i)-\sum_j L^*(R(\alpha_j))\otimes \alpha_j\\
    &{=\mathcal{T}(R,L)+\sum_i \big(L^*(\alpha_i)\circ L^*(R(\alpha_i))\big)\otimes 1+\sum_i 1\otimes \big(\alpha_i\circ R(\alpha_i)\big)}.
\end{align*}
A routine argument using the symmetries of Clifford multiplication and of the curvature operator $R$, including the Bianchi identity, see \cite[Thm.~II.8.8]{lm-book}, implies
\begin{equation*}
    \sum_i \alpha_i\circ R(\alpha_i)=-\tfrac14\operatorname{tr}(R)=-\tfrac18\scal_{R}.
\end{equation*}	
Since \(L^*\circ R\circ L\in\Sym^2(\wedge^2 W)\) has the same symmetries as \(R\), the above also implies
\begin{equation*}
    \sum_i L^*(\alpha_i)\circ L^*(R(\alpha_i))=\sum_i\beta_i\circ(L^*\circ R\circ L(\beta_i))=-\tfrac14\operatorname{tr}(L^*\circ R\circ L).\qedhere
\end{equation*}
\end{proof}

We now estimate the terms in the right-hand side of the identity in \Cref{rt}. The following two lemmas were observed in~\cite[Sec~1.1]{goette-semmelmann2}, but, for completeness, we supply their proofs below using our notations.

\begin{lemma}\label{fscal}
	If \(R\in\Sym^2_b(\wedge^2V)\) is an algebraic curvature operator with $\sec_R\geq0$ and $l\colon W\to V$ is a linear map such that $L=\wedge^2 l$ is nonincreasing, then 
    \begin{equation*}
    \operatorname{tr}(L^*\circ R\circ L)\leq\tfrac12\scal_{R}.
    \end{equation*}
	If, in addition, \(\tfrac12\scal_{R} \succ \Ric_{R}\succ 0\), then equality above implies \(l\) is an isometry.  
\end{lemma}

\begin{proof}
By \Cref{lemma:diag}, we may choose orthonormal bases \(\{v_i\}\) of \(V\) and \(\{w_i\}\) of \(W\) such that \(l(w_i)=\lambda_i v_i\) for some \(\lambda_i\geq0\). 
Then \(\{w_i\wedge w_i\}_{i<j}\) is an orthonormal basis of $\wedge^2 W$ and $(R\circ L)(w_i\wedge w_j)=\lambda_i\lambda_j \, R(w_i\wedge w_j)$ for all $i<j$, so
\begin{equation*}
\operatorname{tr}(L^*\circ R\circ L)=\sum_{i<j}\lambda_i^2\lambda_j^2\, \sec_{R}(v_i\wedge v_j).
\end{equation*}
Since \(L\) is nonincreasing, \(\lambda_i\lambda_j\leq1\) for all \(i < j\), which proves the desired inequality.  

If $\operatorname{tr}(L^*\circ R\circ L)=\tfrac12\scal_{R}$, then
\begin{equation}\label{scaleq}
\sum_{i<j}(1-\lambda_i^2\lambda_j^2)\,\sec_R(v_i\wedge v_j)=0.
\end{equation}
If, furthermore, \(\tfrac12\scal_{R} \succ \Ric_{R}\), then for each fixed \(a\), we have \(\sum_{b}\sec_R(v_a\wedge v_b)<\sum_{i<j}\sec_R(v_i\wedge v_j)\), or equivalently,
\begin{equation*}
    0<\sum_{\substack{i<j\\ i,j\neq a}}\sec_R(v_i\wedge v_j).
\end{equation*}
Thus there exist \(i,j\neq a\) so that \(\sec_R(v_i\wedge v_j)>0\). (Note that \(\tfrac12\scal_{R} \succ \Ric_{R}\) implies $\dim V\geq3$.) As $\sec_R\geq0$, it follows that \eqref{scaleq} implies \(\lambda_i\lambda_j=1\). But \(\lambda_i\lambda_a\leq1\) and \(\lambda_j\lambda_a\leq1\), so we conclude that \(\lambda_a\leq1\) for all \(a\), i.e., \(l\) is nonincreasing. If, moreover, \(\Ric_{R}\succ 0,\) then, for each \(a\), there exists \(b\neq a\) such that \(\sec_R(v_a\wedge v_b)>0\). Again it follows that \(\lambda_a\lambda_b=1\), and we conclude that \(\lambda_a=\lambda_b=1,\) so \(l\) is an isometry.
\end{proof}

\begin{lemma}\label{nneg}
If $R\in\Sym^2(\wedge^2 V)$ is such that \(R\succeq0\), then \(\mathcal{T}(R,L)\succeq0\).
\end{lemma}

\begin{proof}
	Choose an orthonormal basis \(\alpha_i\) of $\wedge^2 V$ that diagonalizes $R$, i.e., such that \(R(\alpha_i)=\rho_i\,\alpha_i\). Since $R\succeq0$, we have that \(\rho_i\geq0\). Then, by \Cref{def:RT},
    \begin{equation*}
    \mathcal{T}(R,L)=-\sum_i\rho_i\left(L^*(\alpha_i)\otimes 1+1\otimes \alpha_i\right)^2.    
    \end{equation*}
	Since $\alpha_i\in\mathfrak{spin}(V)$ and $L^*(\alpha_i)\in\mathfrak{spin}(W)$, i.e., these are elements of $\cl^0(V)$ and $\cl^0(W)$ in the images of the corresponding isomorphisms \eqref{coverderiv}, the endomorphisms $L^*(\alpha_i)\otimes 1+1\otimes \alpha_i$ of $S(W)\otimes S(V)$ are skewsymmetric, so the conclusion follows.
\end{proof}

Let us now assume that both $V$ and $W$ are $4$-dimensional.
We denote the Hodge star operators of $V$ and $W$ by \(*^V\in\Sym^2(\wedge^2V)\) and \(*^W\in\Sym^2(\wedge^2W)\), and similarly for \(\omega_\C^V\in\cl(V)\) and \(\omega_\C^W\in\cl(W)\). Recall from \Cref{subsec:CAandS} and \eqref{coverderiv} that \(\Xi_0^{-1}(*(x\wedge y))=\tfrac12\,\omega_\C \,xy\) if $x$ and $y$ are orthogonal.

\begin{lemma}\label{star}
If the linear map \(l\colon W\to V\) is such that $L=\wedge^2l$ is nonincreasing and \(\dim_\R V=\dim_\R W=4\), then the restriction of \(\mathcal{T}(*^V,L)\) to \(S^+(W)\otimes S^+(V)\) is positive-semidefinite.
\end{lemma}

\begin{proof}
By \Cref{lemma:diag}, we may choose orthonormal bases \(\{v_i\}\) of \(V\) and \(\{w_i\}\) of \(W\) such that \(l(w_i)=\lambda_i v_i\) for some \(\lambda_i\geq0\). The assumption on \(l\) ensures that \(\lambda_i\lambda_j\leq1\) for all \(i\neq j\). We first note that, with these choices,
\begin{equation*}
    L^*\big(*^V(v_i\wedge v_j)\big)=\mu_{ij}*^W (w_i\wedge  w_j),
\end{equation*}
	where \(|\mu_{ij}|=\big\|L^* (*^V (v_i\wedge v_j))\big\|\leq1\), since \(*^V\) is an isometry of $\wedge^2 V$. Symmetries of Clifford multiplication in \(\cl^{0}(W)\otimes \cl^{0}(V)\cong\End\!\big(S^+(W)\otimes S^+(V)\big)\) and the fact that $\Xi_0^{-1}(*(x\wedge y))=\tfrac12\,\omega_\C \,xy$ if $x$ and $y$ are orthonormal imply that:
\begin{align*}
L^*\big(*^V (v_i\wedge v_j)\big)\otimes1+1\otimes*^V (v_i\wedge v_j) &=\tfrac12\left(\mu_{ij}\,\omega^W_\C w_iw_j\otimes 1+1\otimes\omega^V_\C v_i v_j\right)\\
&=\tfrac12\left(\mu_{ij}\,w_iw_j\omega^W_\C\otimes 1+1\otimes v_i v_j\omega^V_\C\right)\\
&=\tfrac12\left(\mu_{ij}\,w_iw_j\otimes 1+1\otimes v_iv_j\right),
\end{align*}
where the last equality holds because the above endomorphisms are restricted to the tensor product $S^+(W)\otimes S^+(V)$ of $+1$-eigenspaces of $\omega_\C^W$ and $\omega_\C^V$.
Therefore,
\begin{align*}
\mathcal{T}(*^V,L) &=-\tfrac14\sum_{i<j}\left(\lambda_i\lambda_j\, w_iw_j\otimes 1+1\otimes v_iv_j\right)\circ\left(\mu_{ij}\, w_iw_j\otimes 1+1\otimes v_iv_j\right)\\
&=\tfrac14\sum_{i<j}\left(-\lambda_i\lambda_j+w_iw_j\otimes v_iv_j\right)\circ\left(-\mu_{ij}+w_iw_j\otimes v_iv_j\right),
\end{align*}
and, since \(\left(w_iw_j\otimes v_iv_j\right)^2=1\), and \(|\lambda_i\lambda_j|\leq1\) as well as \(|\mu_{ij}|\leq1\), we may use a basis of eigenvectors of \(w_iw_j\otimes v_iv_j\) to conclude that the eigenvalues of each of the above summands are \((-\lambda_i\lambda_j+1)(-\mu_{ij}+1)\geq0\) or \((-\lambda_i\lambda_j- 1)(-\mu_{ij}- 1)\geq0.\) 
\end{proof}

\begin{remark}
Under the hypotheses of \Cref{star}, it also follows that the restriction of $\mathcal T(*^V,L)$ to $S^-(W)\otimes S^-(V)$ is \emph{negative}-semidefinite.
\end{remark}

\section{Extremality and rigidity on closed 4-manifolds}\label{4closed}

In this section, we prove \Cref{mainthm1} and Corollary \ref{maincor:CP2CP2} in the Introduction, by showing that a certain twisted Dirac operator has nontrivial kernel and 
using the results of \Cref{pi} to analyze the curvature term in the corresponding Bochner--Lichnerowicz--Weitzenb\"ock formula.

We begin by recalling and generalizing the notions of area-extremality and area-rigidity for scalar curvature discussed in the Introduction to also account for ``topologically modified'' competitors, cf.~Gromov~\cite[Sec.~4]{gromov-dozen} and \cite[Sec.~$5\tfrac49$]{gromov-96}. Henceforth, all manifolds are assumed connected.

\begin{definition}\label{def:extremality-global}
A closed oriented Riemannian manifold $(M,\g_M)$ is \emph{area-extremal with respect to a class $\mathcal C=\{ f\colon (N,\g_N)\to (M,\g_M) \}$} of competitors, consisting of closed oriented Riemannian manifolds $(N,\g_N)$ with $\dim M=\dim N$ and smooth spin maps $f\colon N\to M$ of nonzero degree, if the inequalities
\begin{equation}\label{eq:scal-comp}
\wedge^2\g_N\succeq f^*\wedge^2 \g_M \qquad \text{and}\qquad \scal(\g_N)\geq \scal(\g_M)\circ f	
\end{equation}
imply $\scal(\g_N)=\scal(\g_M)\circ f$. If, in addition, there exists $q\in M$ such that \eqref{eq:scal-comp} implies $\dd f(p)\colon T_pN\to T_{q}M$ is a linear isometry for all competitors $f\colon N\to M$ in $\mathcal C$
and all $p\in f^{-1}(q)$, then $(M,\g_M)$ is called \emph{area-rigid at $q\in M$ with respect to $\mathcal C$}. If $(M,\g_M)$ is area-rigid at all of its points with respect to $\mathcal C$, then it is simply called \emph{area-rigid with respect to $\mathcal C$.}
\end{definition}

Recall that a smooth map \(f\colon N\to M\) is \emph{spin} if it is compatible with second Stiefel--Whitney classes, i.e., $f^*w_2(M)=w_2(N)$, and $\wedge^2\g_N\succeq f^* \wedge^2 \g_M$ means that $f\colon (N,\g_N)\to (M,\g_M)$ is area-nonincreasing, i.e., $\wedge^2 \dd f$ is nonincreasing, namely
\begin{equation*}
	\big\|x\wedge y\big\|_{\g_N}
\geq
	\big\|\dd f(p) x\wedge \dd f(p) y\big\|_{\g_M}
\end{equation*}
for all $x,y\in T_{p}N$ and all $p\in N$. For example, this holds if $f\colon (N,\g_N)\to (M,\g_M)$ is distance-nonincreasing, see \Cref{subsec:linalgebra}.

\begin{remark}\label{rem:f=id}
The notions of area-extremality and area-rigidity for closed manifolds in the Introduction correspond to using the class $\mathcal C^{\id}_0:=\{\id\colon (M,\g_1)\to (M,\g_0)\}$ in \Cref{def:extremality-global}, cf.~\eqref{eq:g1-competitor-global}. Note that competitors given by any diffeomorphisms $f\colon (M,\g_1)\to (M,\g_0)$ reduce to the above case, pulling back $\g_1$ by $f^{-1}$.
\end{remark}

\subsection{Index theory}\label{indextheory}
Let $(M,\g_M)$ and $(N,\g_N)$ be closed oriented Riemannian $4$-manifolds,  and denote by \(S(TM)\) and \(S(TN)\) their (locally defined) spinor bundles. 
Given a spin map $f\colon N\to M$, the twisted spinor bundle \(S(TN)\otimes f^*S(TM)\) is globally defined; namely, it is the spinor bundle of the spin bundle \(TN\oplus f^*TM\).

Let \(E=f^*S^+(TM)\), and consider the twisted Dirac operator
\begin{equation}\label{eq:DE-global}
D_E\colon \Gamma(S(TN)\otimes E)\longrightarrow\Gamma(S(TN)\otimes E).	
\end{equation}
Recall from \eqref{eq:twisted-Dirac} that, 
if we denote by $\nabla^{S_N}$ and $\nabla^{S_M}$ the connections on $S(TN)$ and $S(TM)$ respectively, and by \(\{e_i\}\) a local $\g_N$-orthonormal frame for \(TN\), then on a decomposable local section \(\phi\otimes f^*\psi\) of \(S(TN)\otimes E\), the operator \(D_E\) acts as
\[D_E(\phi\otimes \psi)=\sum_{i=1}^4 \left(e_i\nabla^{S_N}_{e_i}\phi\right)\otimes f^*\psi+(e_i\phi)\otimes f^*\left(\nabla^{S_M}_{\dd f(e_i)}\psi\right) .\]

We respectively denote by \(\chi(X)\) and \(\sigma(X)\) the Euler characteristic and signature of an oriented $4$-manifold \(X\), and by $\deg(f)$ the degree of $f\colon N\to M$.

\begin{lemma}\label{kernel2}
If $f\colon N\to M$ has $\deg(f)\neq0$ and
\begin{equation}\label{eq:nontrivialKerDE}
2\chi(M)+3\sigma(M)>\frac{\sigma(N)}{\deg(f)},	
\end{equation}
then the restriction of \(D_E\) to \(\Gamma(S^+(TN)\otimes E)\) has nontrivial kernel.
\end{lemma}

\begin{proof}
By the splitting principle, we may assume that \(TM\otimes\C\cong \lambda_1\oplus\overline\lambda_1\oplus\lambda_2\oplus\overline\lambda_2\), where \(\lambda_1,\lambda_2\) are complex line bundles. Let \(x_1,x_2\) denote the first Chern classes of \(\lambda_1,\lambda_2\) respectively. If \(\lambda_1\) and \(\lambda_2\) are spin, \(S^+(TM)\cong \lambda_1^{\frac12}\otimes \lambda_2^{\frac12}\oplus {\overline \lambda}_1^{\frac12}\otimes{\overline \lambda}_2^{\frac12}\), cf.~\cite[p.~238]{lm-book}, and the Chern character of $S^+(TM)$ is given by
\[\mathrm{ch}(S^+(TM))=2\cosh\left({\tfrac{x_1}2+\tfrac{x_2}2}\right)=2+\tfrac14 p_1(TM)+\tfrac12 e(TM),\]
where \(p_1\) is the first Pontryagin class and \(e\) is the Euler class. By an argument in \cite[App.~A4]{shanahan}, the same formula also holds in the case that \(\lambda_1\) and \(\lambda_2\) are not spin. Then, by the Atiyah--Singer Index Theorem,
\begin{align*}
\text{ind}\big(D_E|_{S^+(TN)\otimes E}\big)&=\left<\hat{A}(TN)\cdot f^*\mathrm{ch}(S^+(TM)),[N]\right>\\&=\left<-\tfrac{1}{12}p_1(TN)+\tfrac14 f^*p_1(TM) +\tfrac12 f^*e(TM),[N]\right>\\&=-\tfrac14\sigma(N)+\text{deg}(f)\left(\tfrac34\sigma(M)+\tfrac12\chi(M)\right)>0,
\end{align*}
and hence \(\ker\big(D_E|_{S^+(TN)\otimes E}\big)\) is nontrivial.
\end{proof}

Note that all hypotheses of \Cref{kernel2} are satisfied if $\deg(f)=1$ and
$M=N$ has vanishing first Betti number $b_1(M)=0$, e.g., if $M$ is simply-connected, since then
$\chi(M)=2+b_+(M)+b_-(M)$ and $\sigma(M)=b_+(M)-b_-(M)$, where $b_\pm(M)$ are the self-dual/anti-self-dual second Betti numbers, so \eqref{eq:nontrivialKerDE} simplifies to $4+4b_+(M)>0$.

\subsection{Extremality and Rigidity}
We now combine the pointwise inequalities from \Cref{pi} with the Bochner--Lichnerowicz--Weitzenb\"ock formula \eqref{twistlich} for the twisted Dirac operator \eqref{eq:DE-global} and \Cref{kernel2} to prove our main result on area-extremality and area-rigidity of closed $4$-manifolds, in the sense of \Cref{def:extremality-global}.

\begin{theorem}\label{thm}
A closed oriented Riemannian $4$-manifold \((M,\g_M)\) with $\sec\geq0$ such  that $R+\tau\,*\succeq0$ for a nonpositive $\tau\colon M\to \R$ is area-extremal with respect to
\begin{equation*} 
\mathcal C_0:=\left\{f\colon (N,\g_N)\to (M,\g_M) \; : \; 2\chi(M)+3\sigma(M)>\frac{\sigma(N)}{\deg(f)}\right\}.
\end{equation*}
If, in addition, $\frac{\scal(\g_M)}{2}\g_M\succ\Ric(\g_M)\succ0$ at $q\in M$, then $(M,\g_M)$ is area-rigid at $q\in M$ with respect to $\mathcal C_0$.
\end{theorem}

\begin{proof}
Let $f\colon (N,\g_N)\to (M,\g_M)$ be a competitor in $\mathcal C_0$, and let $E=f^*S^+(TM)$.
The Bochner--Lichnerowicz--Weitzenb\"ock formula for the square \(D_E^2\) of the twisted Dirac operator \eqref{eq:DE-global} is given by \eqref{twistlich}, setting \( \g=\g_N\). From \eqref{eq:curvop-sign} and \eqref{spinorcurv},
 \[R^{{S_M}}_{x,y}=-\Xi_0^{-1} \big( R(x\wedge y)\big),\]
where $R\colon \wedge^2 TM\to\wedge^2 TM$ is the curvature operator of $(M,\g_M)$.
Working pointwise with $W=T_pN$, $V=T_{f(p)}M$, $l\colon W\to V$ given by $\dd f(p)$, and  $L=\wedge^2 l$, we see that \eqref{twistlich} can be written using the endomorphism $\mathcal R(R,L)$ from \Cref{def:RT} as
\begin{equation}\label{d2}
  D_E^2=\nabla^*\nabla+\tfrac{1}{4}\scal(\g_N)+\mathcal{R}(R, L).
\end{equation}
Let \(\xi\in\Gamma(S^+(TN)\otimes E)\) be a nonzero section in the kernel of \(D_E\), which exists by \Cref{kernel2}. Applying \eqref{d2} and integrating, we obtain, using \Cref{rt},
\begin{equation}\label{eq:pfeq1}
\begin{aligned}
0&=\int_{N} \langle \nabla^*\nabla\xi, \xi\rangle+\tfrac14\int_N \scal(\g_N)\|\xi\|^2+\int_N\left<\mathcal{R}(R,L)\xi,\xi\right>\\
&=\int_{N}\|\nabla \xi\|^2+\int_N\left<\mathcal{T}(R,L)\xi,\xi\right> \\
&\quad +\tfrac14\int_N\left( \scal(\g_N)-\tfrac12\scal(\g_M)\circ f-\operatorname{tr}(L^*\circ R \circ L)\circ f\right)\|\xi\|^2.
\end{aligned}
\end{equation}
Linearity of $R\mapsto \mathcal{T}(R,L)$ and \Cref{nneg,star} imply that, on \(S^+(TN)\otimes E\),
\begin{equation}\label{eq:pfeq2}
\mathcal{T}(R,L)=\mathcal{T}(R+\tau\, *,L) -\tau \;\mathcal{T}(*,L)\succeq0,
\end{equation}
since $\tau\leq 0$.
Thus, combining \eqref{eq:scal-comp}, \eqref{eq:pfeq1}, and \eqref{eq:pfeq2} with \Cref{fscal}, which may be applied because $\sec_R\geq0$ and $L$ is nonincreasing by \eqref{eq:scal-comp}, we conclude that 
\begin{align*}
0&\geq\int_{N}\|\nabla \xi\|^2+\tfrac14\int_N\left(\scal(\g_N)-\tfrac12\scal(\g_M)\circ f-\operatorname{tr}(L^*\circ R \circ L)\circ f\right)\|\xi\|^2\\
&\geq\int_{N}\|\nabla \xi\|^2+\tfrac14\int_N\left(\scal(\g_N)-\scal(\g_M)\circ f\right)\|\xi\|^2\geq0.
\end{align*}
Therefore, \(\nabla\xi\) vanishes identically and hence \(\|\xi\|\neq0\) is constant, so it follows that $\scal(\g_N)=\scal(\g_M)\circ f$ everywhere.  Furthermore, \(\operatorname{tr}(L^*\circ R \circ L)=\tfrac12\scal(\g_M)\), so, if $\frac{\scal(\g_M)}{2}\g_M\succ\Ric(\g_M)\succ0$ at $q=f(p)$, then \(l\) is an isometry by \Cref{fscal}.
\end{proof}

Let us briefly discuss some situations in which the hypotheses of \Cref{thm} are satisfied. If \((M,\g_M)\) is simply-connected and its curvature operator $R$ satisfies $R+\tau\,*\succeq0$ with $\tau\leq0$, then, by the proof of \cite[Thm.~D]{bm-iumj}, either \((M,\g_M)\) is isometric to \(\S^2\times \S^2\) and $\tau\equiv0$, hence it is area-extremal with respect to any class of competitors by \cite{goette-semmelmann2}, or else $M$ has negative-definite intersection form, i.e., $b_+(M)=0$, and hence \(\sigma(M)=-b_-(M)=-b_2(M)\). In this latter case, restricting ourselves to \emph{self-map competitors}, i.e., setting \(N=M\), the class $\mathcal C_0$ simplifies to
\begin{equation*}
\mathcal C_0^{\mathrm{self}}:=\left\{f\colon (M,\g_1)\to (M,\g_0) \; : \; 4+\left(\frac{1}{\deg(f)}-1\right)b_2(M)>0 \right\}.
\end{equation*}
Clearly, $\mathcal C^{\id}_0\subset\mathcal C_0^{\mathrm{self}}$, and this proves \Cref{mainthm1} in the Introduction, since we may choose the orientation of $M$ for which $\tau\leq 0$ in order to apply \Cref{thm}, see \Cref{rem:f=id}.
Moreover, if \(b_2(M)\leq4\), then $\mathcal C_0^{\mathrm{self}}$ contains all (smooth spin) self-maps $f\colon (M,\g_1)\to (M,\g_0)$ of nonzero degree. Thus, more generally, we have:

\begin{corollary}\label{cor:self-maps}
A closed simply-connected Riemannian $4$-manifold $(M,\g_M)$ whose curvature operator $R$ satisfies $R+\tau\,*\succeq0$ with $\tau\leq0$ is area-extremal with respect to any (smooth spin) self-maps of degree $1$, and (smooth spin) self-maps of arbitrary nonzero degree if \(b_2(M)\leq 4\).
If, in addition, $\frac{\scal(\g_M)}{2}\g_M\succ\Ric(\g_M)\succ0$ at $q\in M$, then $(M,\g_M)$ is area-rigid at $q\in M$ with respect to the same classes of self-maps.
\end{corollary}

\subsection{Examples}\label{examples}
The only closed simply-connected $4$-manifolds currently known to admit  metrics with $\sec\geq0$ are $\S^4$, $\C P^2$, $\S^2\times \S^2$, $\C P^2\#\C P^2$, and $\C P^2\#\overline{\C P^2}$, and this list is conjectured to be exhaustive. We now analyze the existence of families of area-extremal and area-rigid metrics on these manifolds, starting with those known to admit $\sec>0$.

\begin{theorem}\label{thm:S4RP4CP2}
    The spaces of Riemannian metrics on $\S^4$ and $\C P^2$ have nonempty open subsets     consisting of area-rigid metrics with respect to self-maps of any nonzero degree, that contain the round metric and the Fubini--Study metric, respectively.
\end{theorem}

\begin{proof}
Any metric on \(\S^4\) with $R\succ0$ 
is area-rigid with respect to self-maps of any nonzero degree by \Cref{cor:self-maps} (taking $\tau\equiv0$), or by \cite{goette-semmelmann2}. In particular, there is a neighborhood of the round metric consisting of such metrics.

Let \(\g_{\C P^2}\) be the Fubini--Study metric on \(\C P^2\) and denote by $R_{\C P^2}$ its curvature operator, with $1\leq \sec\leq 4$. In a self-dual/anti-self-dual basis of $\wedge^2 T_p\C P^2$, i.e., a basis that diagonalizes the Hodge star operator as $*_{\C P^2}=\diag(1,1,1,-1,-1,-1)$, we have $R_{\C P^2}=\diag(0,0,6,2,2,2)$, so $R_{\C P^2}+\tau\, *_{\C P^2}\succeq0$ if and only if $0\leq \tau\leq 2$. Thus, to apply \Cref{thm}, we must set $M$ to be $\overline{\C P^2}$, i.e., $\C P^2$ with the opposite orientation, which is isometric to $\C P^2$ and hence has the same curvature operator $R_{\overline{\C P^2}}=R_{\C P^2}$, but $*_{\overline{\C P^2}}=-*_{\C P^2}$, so $R_{\overline{\C P^2}}+\tau\,*_{\overline{\C P^2}}\succ0$ if and only if $-2<\tau<0$. We may then take, e.g., $\tau\equiv-1$, and, by continuity, there exists a neighborhood \(\mathcal U\) of \(\g_{\C P^2}\) in the $C^2$-topology formed by metrics $\g$ that also satisfy $R_\g+\tau\,*_\g\succ0$ for the constant function $\tau\equiv-1$. Furthermore, 
since $\g_{\C P^2}$ is an Einstein metric,
\[\frac{\scal(\g_{\C P^2})}{2}\g_{\C P^2}\succ\frac{\scal(\g_{\C P^2})}{4}\g_{\C P^2}=\Ric(\g_{\C P^2})\succ0,\]
so, up to shrinking $\mathcal U$, we may assume that all \(\g\in\mathcal U\) satisfy $\frac{\scal(\g)}{2}\g\succ\Ric(\g)\succ0$. Therefore, since $b_2(\C P^2)=1$, it follows from \Cref{cor:self-maps} that \((\C P^2,\g)\) is area-rigid with respect to self-maps of any nonzero degree, for all $\g\in\mathcal U$.
\end{proof}

Let us consider the remaining simply-connected $4$-manifolds known to admit $\sec\geq0$. On $\S^2\times \S^2$, any product of metrics  on $\S^2$ with $\sec>0$ has $R\succeq0$ and $\frac{\scal(\g)}{2}\g \succ\Ric(\g)\succ0$, thus it is area-rigid with respect to self-maps of any nonzero degree by \Cref{cor:self-maps}; or, in a more general sense, by \cite{goette-semmelmann2}.
On \({\C}P^2\#\overline{{\C}P}^2\) and \({\C}P^2\#{\C}P^2\), metrics with $\sec\geq0$ were constructed by Cheeger \cite{cheeger}, using a gluing method that 
endows the connected sum $M_1\# M_2$ of any two (oriented) compact $n$-dimensional rank one symmetric spaces with $\sec\geq0$. These \emph{Cheeger metrics} contain a \emph{neck} region isometric to a round cylinder $\S^{n-1}\times [0,\ell]$, to which the complement of a ball in each $M_i$ is glued. The length $\ell\geq0$ of the neck can be chosen arbitrarily, including $\ell=0$.
There do not exist metrics on \({\C}P^2\#\overline{{\C}P}^2\) that satisfy the hypotheses of \Cref{thm}, as a consequence of \cite[Thm.~D]{bm-iumj}. However, \({\C}P^2\#\overline{{\C}P}^2\) admits K\"ahler metrics of positive Ricci curvature \cite{yau-calabiconj}, 
which are area-rigid by \cite{goette-semmelmann1}.
Meanwhile, existence of area-extremal or area-rigid metrics on \({\C}P^2\#{\C}P^2\) cannot be established with any previous results, since they either require metrics that are K\"ahler or have positive-semidefinite curvature operator, and \({\C}P^2\#{\C}P^2\) admits neither. We overcome these restrictions as follows: 

\begin{theorem}\label{thm:cheeger-metrics}
The Cheeger metrics on \({\C}P^2\#{\C}P^2\) are area-extremal with respect to  self-maps of any nonzero degree, and area-rigid with respect to these maps on all points
in the complement ${\C}P^2\#{\C}P^2 \setminus (\S^3\times [0,\ell])$ of the neck region.
\end{theorem}

\begin{proof}
The manifold \({\C}P^2\#{\C}P^2\) is obtained gluing together two equally oriented copies of the normal disk bundle $\nu(\C P^1)$ of $\C P^1\subset\C P^2$ along their  boundary $\S^3$, with an orientation-reversing diffeomorphism $\alpha$. 
Since bi-invariant metrics on $\mathsf{SU}(2)\cong\S^3$ are round (in particular, they support orientation-reversing isometries), Cheeger metrics on each copy of $\nu(\C P^1)$ are isometric to the metrics with $\sec\geq0$ constructed by Grove and Ziller~\cite{grove-ziller-annals} on $\nu(\C P^1)\cong \mathsf G\times_{\mathsf K} D^2$ that are invariant under the action of $\mathsf G=\mathsf{SU}(2)$ with orbit space $\nu(\C P^1)/\mathsf G=\left[0,r_{\mathrm{max}}\right]$, where $0$ corresponds to the singular orbit with isotropy $\mathsf K=\S^1$ and all other (principal) orbits have trivial isotropy $\mathsf H=\{1\}$.
A detailed treatment of Grove--Ziller metrics on a related disk bundle $\mathsf G'\times_{\mathsf K'} D^2$ over $\C P^1=\mathsf G/\mathsf K=\mathsf G'/\mathsf K'$
is given in \cite{bk-rf}, where $\mathsf G'=\mathsf{SO}(3)$, $\mathsf K'=\mathsf{SO}(2)_{1,2}$, and $\mathsf H'=\Z_2$, including the explicit computation of their curvature operator, see \cite[Prop.~3.5]{bk-rf}. This computation applies mutatis mutandis to $\mathsf G\times_{\mathsf K} D^2$, as the Lie groups $(\mathsf G,\mathsf K,\mathsf H)$ and $(\mathsf G',\mathsf K',\mathsf H')$ have isomorphic Lie algebras. Namely, one must only replace the constant $\rho(b)=b/2$ with $\rho(b)=2b$ when passing from $\mathsf G'\times_{\mathsf K'} D^2$ to $\mathsf G\times_{\mathsf K} D^2$. 
Then \cite[Prop.~3.5]{bk-rf} implies that, in a basis induced by an orthonormal frame, the curvature operator $R$ of a Grove--Ziller metric $\g$ on $\nu(\C P^1)\cong \mathsf G\times_{\mathsf K} D^2$ is block diagonal, i.e., $R=\diag(R_1,R_2,R_3)$, where
 \begin{equation}\label{eq:curvopGZ}
 R_1 = \begin{bmatrix}
           \frac{4b^2 - 3\varphi^2}{4b^4} & -\frac{\varphi'}{b^2} \\
           -\frac{\varphi'}{b^2} & -\frac{\varphi''}{\varphi}
          \end{bmatrix}, \quad 
  R_2 = R_3 = \begin{bmatrix}
           \frac{\varphi^2}{4b^4} & \frac{\varphi'}{2b^2} \\
           \frac{\varphi'}{2b^2} & 0
          \end{bmatrix},
 \end{equation}
$b>0$ is an arbitrary constant, $\varphi\colon \left[0,r_{\mathrm{max}}\right]\to\R$ is a nonnegative smooth function satisfying $\varphi(0)=0$, $\varphi'(0)=1/2$, $\varphi^{(2k)}(0)=0$ for all $k\in\mathds N$, 
$\varphi'(r) \geq 0$ for all $r\in \left[0,r_{\mathrm{max}}\right]$,
$\varphi''(r)\leq 0$ for all $r\in\left[0,r_{\mathrm{max}}\right]$, and $\varphi(r)\equiv b$ for all $r\in \left[r_0,r_{\mathrm{max}}\right]$, with $0<r_0\leq r_{\mathrm{max}}$ and $r_0$ chosen sufficiently large. 

Clearly, all data on $(\nu(\C P^1),\g)$ which is invariant under isometries, such as $R$, is determined by its value along a unit speed horizontal geodesic parametrized on $\left[0,r_{\mathrm{max}}\right]$, and $(\nu(\C P^1),\g)$ is isometric to the round cylinder $\S^3(2b)\times [r_0,r_{\mathrm{max}}]$ near its boundary, where $\S^3(2b)$ is a round $3$-sphere of radius~$2b$. Thus, gluing two copies of $(\nu(\C P^1),\g)$ along $\S^3(2b)$ using an orientation-reversing isometry $\alpha$ produces a smooth metric on 
\({\C}P^2\#{\C}P^2=\nu(\C P^1)\cup_{\alpha} \nu(\C P^1)\), which we also denote by $\g$.
The preimage of $\left[r_0,r_{\mathrm{max}}\right]$ under the projection map $\nu(\C P^1)\to \left[0,r_{\mathrm{max}}\right]$ is half of the neck region in \({\C}P^2\#{\C}P^2\); in particular, the neck region is isometric to $\S^3(2b)\times [0,\ell]$, where  $\ell=2(r_{\mathrm{max}}-r_0)\geq0$. Note that while $r_0>0$ must be chosen sufficiently large in order for $\varphi$ as above to exist, the value of $r_{\mathrm{max}}\geq r_0$ is arbitrary.

The matrix of the Hodge star operator $*$ is also block diagonal in the basis used in \eqref{eq:curvopGZ}, namely $* = \operatorname{diag}(H, H, H)$, where 
\begin{equation*}
H= \begin{bmatrix}
      0 & 1\\
      1 & 0
     \end{bmatrix}.
\end{equation*}
Therefore, the unique function $\tau\colon \nu(\C P^1)\to\R$ such that $R+\tau\,*\succeq0$ is $\tau=-\frac{\varphi'}{2b^2}$, and hence the unique function $\tau \colon \C P^2\# \C P^2\to\R$ such that $R+\tau\,*\succeq0$ satisfies $\tau\leq 0$ and $\tau=0$ along the neck region. 
Routine computations with \eqref{eq:curvopGZ} yield
\begin{equation*}
\begin{aligned}
\Ric(\g)&=\operatorname{diag}\left(\tfrac{\varphi^2}{2b^4}-\tfrac{\varphi''}{\varphi},\; \tfrac{1}{b^2}-\tfrac{\varphi^2}{2b^4},\; \tfrac{1}{b^2}-\tfrac{\varphi^2}{2b^4}, \;-\tfrac{\varphi''}{\varphi} \right),\\
\scal(\g)&=\tfrac{2}{b^2}-\tfrac{\varphi^2}{2b^4}-2\tfrac{\varphi''}{\varphi},\\
\tfrac{\scal(\g)}{2}\g  - \Ric(\g)&=\operatorname{diag}\left(
\tfrac{4b^2-3\varphi^2}{4b^4},\; \tfrac{\varphi^2}{4b^4}-\tfrac{\varphi''}{\varphi},\;\tfrac{\varphi^2}{4b^4}-\tfrac{\varphi''}{\varphi},\;\tfrac{4b^2-\varphi^2}{4b^4}\right),
\end{aligned}
\end{equation*}
for all $r\in [0,r_{\mathrm{max}}]$.
In particular, $\tfrac{\scal(\g)}{2}\g \succ \Ric(\g) \succeq0$ on all of $\C P^2\# \C P^2$, and $\Ric(\g)\succ0$  outside the neck region.
Thus, as $b_2(\C P^2\# \C P^2)=2$,  \Cref{cor:self-maps} implies that $\g$ is area-extremal with respect to self-maps of any nonzero degree, and area-rigid with respect to these maps on all points outside the neck region.
\end{proof}

\begin{remark}\label{rem:noneck}
The statement concerning area-rigidity in \Cref{thm:cheeger-metrics} is sharp. In fact, given a Cheeger metric $\g_0$ on $\C P^2\# \C P^2$ with neck length $\ell_0\geq0$, one can easily produce nonisometric competitor Cheeger metrics $\g_1$ that satisfy \eqref{eq:g1-competitor-global} by elongating the neck length to any $\ell_1>\ell_0$. However, Cheeger metrics with \emph{vanishing} neck length are area-rigid (at all points) with respect to \emph{diffeomorphisms}, since any {diffeomorphism which is an} isometry {on} the complement $\C P^2\# \C P^2\setminus(\S^3\times\{0\})$ of the neck hypersurface, which is dense, must be an isometry on all of $\C P^2\# \C P^2$. 
\end{remark}

\Cref{thm:S4RP4CP2,thm:cheeger-metrics} and \Cref{rem:noneck} imply \Cref{maincor:CP2CP2} in the Introduction.

\begin{remark}
Although both halves of \(\nu(\C P^1)\cup_{\alpha} \nu(\C P^1)={\C}P^2\#{\C}P^2\) have cohomogeneity one $\mathsf{SU}(2)$-actions, these actions do not extend to \({\C}P^2\#{\C}P^2\). Indeed, the gluing isometry $\alpha\colon \S^3(2b)\to\S^3(2b)$ is orientation-reversing and hence cannot be $\mathsf{SU}(2)$-equivariant. Using an orientation-\emph{preserving} isometry $\beta$ instead, the resulting manifold is $\nu(\C P^1)\cup_{\beta} \nu(\C P^1)=\C P^2\# \overline{\C P^2}$, which carries a cohomogeneity one $\mathsf{SU}(2)$-action and invariant Cheeger metrics with $\sec\geq0$. Nevertheless, the corresponding function $\tau$ assumes \emph{opposite} signs on each half due to the flip in orientation, hence \Cref{thm} and \Cref{cor:self-maps} do not apply, cf.~\cite[Thm.~D]{bm-iumj}.
\end{remark}

\section{Extremality and rigidity on 4-manifolds with boundary}\label{4boundary}

In this section, we prove Theorems \ref{mainthm-local1} and \ref{mainthm-local2} in the Introduction, adapting the approach from the previous section to the case of compact manifolds with boundary.
We begin by generalizing the notions of area-extremality and area-rigidity for closed manifolds in \Cref{def:extremality-global} to manifolds with boundary, along the lines of \cite{lott-spin}. 

Given a compact Riemannian manifold $(M,\g)$ with boundary $\partial M$, we denote by $\II_{\partial M}$ the second fundamental form of $\partial M$ with respect to the inward unit normal, and by $H(\g)=\operatorname{tr} \II_{\partial M}$ the mean curvature of $\partial M$ computed accordingly.

\begin{definition}\label{def:extremality-local}
	A compact oriented Riemannian manifold $(M,\g_M)$ with boundary is \emph{area-extremal with respect to a class $\mathcal C=\{ f\colon (N,\g_N)\to (M,\g_M) \}$} of competitors, consisting of compact oriented Riemannian manifolds $(N,\g_N)$ with boundary such that $\dim M=\dim N$, and smooth boundary-preserving spin maps $f\colon N\to M$ of nonzero degree, if the inequalities
	\begin{equation}\label{eq:scal-comp-local}
	\begin{aligned}
	&\wedge^2\g_N\succeq f^*\wedge^2 \g_M, &\quad \scal(\g_N)&\geq \scal(\g_M)\circ f,\\
	&\g_N|_{\partial N}\succeq f^* \g_M|_{\partial M}, &\quad H(\g_N)&\geq H(\g_M)\circ f
	\end{aligned}
	\end{equation}
	imply $\scal(\g_N)=\scal(\g_M)\circ f$ and \(H(\g_N)= H(\g_M)\circ f\). If, in addition, there exists $q\in M$ such that \eqref{eq:scal-comp-local} implies $\dd f(p)\colon T_pN\to T_{q}M$ is a linear isometry for all competitors $f\colon N\to M$ in $\mathcal C$ and all $p\in f^{-1}(q)$, then $(M,\g_M)$ is called \emph{area-rigid at $q\in M$ with respect to $\mathcal C$}. If $(M,\g_M)$ is area-rigid at all of its points with respect to $\mathcal C$, then it is simply called \emph{area-rigid with respect to $\mathcal C$.}
\end{definition}

The notions of area-extremality and area-rigidity for manifolds with boundary discussed in the Introduction correspond to using the class
	\[\mathcal C^{\text{id}}_\partial =\big\{ \id\colon (M,\g_1)\to (M,\g_0) : \g_1|_{\partial M}=\g_0|_{\partial M}\big\}.\]
Note that in the case $\partial M=\emptyset$, the class $\mathcal C^{\text{id}}_\partial$ is simply $\mathcal C^{\text{id}}_0$ from \Cref{rem:f=id}.

\subsection*{Index Theory}
Let \((M,\g_M)\) and \((N,\g_N)\) be oriented Riemannian $4$-manifolds as in \Cref{4closed}, except here we allow these manifolds to have nonempty boundary and assume \(f|_{\partial N}\colon \partial N\to \partial M\) is an orientation-preserving isometry, which we also refer to as an \emph{oriented isometry} for shortness.

Let \(E=f^*S^+(TM)\).  As described in \Cref{4closed}, \(\g_N\) and \(\g_M\) induce connections \(\nabla^{S_N}\) and \(\nabla^{S_M}\) on \(S(TN)\) and \(S(TM)\), respectively.  The connections \(\nabla^{S_N}\) and \(f^*\nabla^{S_M}\) then induce a connection \(\nabla\) on \(S(TN)\otimes E\), which is used to define the twisted Dirac operator \(D_E\) on sections of that bundle.  While \(\nabla\) restricts to a connection on \(S^+(TN)\otimes E|_{\partial N}\), we will next define another connection \(\nabla^\partial\) on that bundle, which is induced by the boundary metrics  \(\g_N|_{\partial N}\) and \(\g_M|_{\partial M}\), rather than the ambient metrics \(\g_N\) and \(\g_M\), and thus differs from \(\nabla\) by a term related to the second fundamental form of the boundaries, see \eqref{sff}.  

Let \(r\colon N\to\R\) be the distance from \(\partial N\) and choose \(\varepsilon>0\) smaller than the focal radius of $\partial N$, so that the normal exponential map of \(\partial N\) is a diffeomorphism from \(\partial N\times [0,\varepsilon]\) to \(U=\{x\in N : r(x)\leq \varepsilon\}\). Let \(\nu=\partial _r,\) so \(\nu|_{\partial N}\) is the inward unit normal.  There is an isomorphism \[\cl(T\partial N)\longrightarrow \cl^0(TN)|_{\partial N}\] generated by \( T\partial N\ni v \mapsto -\nu v\in\cl(TN).\) Using this isomorphism, we identify \(S(T\partial N)\) and \(S^+(TN)|_{\partial N}\).  Let \(\nabla^{\partial N}\) be the connection on \(S(T\partial N)=S^+(TN)|_{\partial N}\) induced by the restriction of $\g_N$ to \(\partial N\).  
Similarly, let \(\nabla^{\partial M}\) be the connection induced by $\g_M|_{\partial M}$ on \(S(T\partial M)=S^+(TM)|_{\partial M}.\)  The connections \(\nabla^{\partial N}\) and \(f^*\nabla^{\partial M}\) induce a connection \(\nabla^\partial \) on \(S^+(TN)\otimes E|_{\partial N}\).

Using parallel translation along normal geodesics, we can identify
\begin{equation}\label{collar}
\Gamma\left(S(TN)\otimes E|_U\right)\cong C^\infty\big([0,\varepsilon],\,\Gamma\left(S(TN)\otimes E|_{\partial N}\right)\!\big),
\end{equation}
that is, sections of $S(TN)\otimes E$ on the collar neighborhood $U$ are $1$-parameter families of sections of $S(TN)\otimes E$ on the boundary $\partial N$.
Using this identification, we have
\[D_E|_{\Gamma\left(S^+(TN)\otimes E|_U\right)}=\nu\left(\partial_r+B+ A\right),\]
where, in a local orthonormal frame \(\{e_i\}\) for \(TN\) such that \(e_1=-\nu,\)
\begin{equation}\label{eq:BtildeA}
B=\sum_{i=2}^4e_1e_i\nabla^\partial_{e_i}, \qquad A=\sum_{i=2}^4e_1e_i\left(\nabla_{e_i}-\nabla^\partial_{e_i}\right).
\end{equation}
Recall that Clifford multiplication is only on the first factor of \(S(TN)\otimes E\), so \(B\) is a first-order self-adjoint elliptic differential operator on \(S(TN)\otimes E|_{\partial N}\) and \(A\) is a self-adjoint linear endomorphism of the same bundle.  

Let \(P_{>0}\) and \(P_{\leq0}\) be the projections onto the subspaces of \(\Gamma(S^+(TN)\otimes E|_{\partial N})\) spanned by eigenvectors of \(B\) with positive and nonpositive eigenvalues, respectively. We impose boundary conditions on \(D_E\) by setting $D_E^+:=D_E|_{\Gamma^+}$, where
\begin{equation}\label{eq:gamma+def}
\Gamma^+:=\{\xi\in\Gamma(S^+(TN)\otimes E) : P_{>0}\left(\xi|_{\partial N}\right)=0\}.	
\end{equation}
The adjoint $D_E^-$ of \(D_E^+\) is the restriction of \(D_E\) to \[\Gamma^-:=\{\xi\in\Gamma(S^-(TN)\otimes E) : P_{\leq0}\left(\nu\xi|_{\partial N}\right)=0\}.\] 
We use the convention that \(\sigma(M)\) is the signature of the bilinear form induced by the cup product on the image of \(H^2(M,\partial M)\) in \(H^2(M)\), and similarly for $\sigma(N)$.

\begin{lemma}\label{lem:section}
    If \(2\chi(M)+3\sigma(M)+2b_0(\partial M)+2b_2(\partial M)>\sigma(N)\), then $D_E^+$  has nontrivial kernel.
\end{lemma}   

\begin{proof}
With the boundary conditions given by $\Gamma^+$, we claim that the index of the operator \(D_E^+\) is
\begin{equation}\label{eq:ind-bdy}
	\text{ind}\big(D_E^+\big)=\tfrac14\big(-\sigma(N)+2\chi(M)+3\sigma(M)+2b_0(\partial M)+2b_2(\partial M)\big),
\end{equation}
which implies the desired result, since \(\text{ind}(D_E^+)=\dim\ker(D_E^+)-\dim\ker(D_E^-).\)

Using the identification \eqref{collar}, the formal adjoint of \(D_E^+\) on $U$ is 
\[D_E^-=(-\partial_r+B+A)(-\nu)=\nu(\partial_r+\nu B\nu+\nu A\nu).\]
This operator satisfies assumptions (2.2) in \cite{Grubb}, with \(P\), $\sigma_0$, \(A,\) and \(\psi\) in the notation of \cite{Grubb} given by \(D_E^-\), \(\nu\), \(\nu B\nu\), and \(\nu A\nu|_{\partial N}\), respectively, in our notation.  

Let \(P'_{\geq0}\) and \(P'_{<0}\) be the projections onto the subspaces of \(\Gamma(S^-(TN)\otimes E|_{\partial N})\) spanned by eigenvectors of \(\nu B\nu\) with nonnegative and negative eigenvalues, respectively. Since \(\nu\) is skewsymmetric, we have that
\begin{equation*}
\begin{aligned}
\Gamma^-&=\{\xi\in\Gamma(S^-(TN)\otimes E)  : P'_{\geq0}(\xi|_{\partial N})=0\},\\
\Gamma^+&=\{\xi\in\Gamma(S^+(TN)\otimes E) : P'_{<0}\left(\nu\xi|_{\partial N}\right)=0\}.
\end{aligned}
\end{equation*}
These spaces define exactly the boundary conditions given by (2.8) in \cite{Grubb}. Thus, by \cite[Lemma 2.2]{Grubb}, we can realize \(D_E^-\) with boundary conditions \(\Gamma^-\) as a Fredholm operator, with adjoint given by a realization of \(D_E^+\) with boundary conditions \(\Gamma^+\). Furthermore, by \cite[Thm.~2.3]{Grubb}, and the discussion which follows (see in particular the paragraph preceding Corollary 2.4), the index of \(D_E^-\) with these boundary conditions is constant under a deformation that fixes the boundary conditions \(\Gamma^\pm\) and induces a continuous deformation of the principal symbol of \(D_E^-\). In particular, a smooth deformation of the map $f\colon N\to M$ and of the metrics $\g_M$ and $\g_N$ that keeps the boundary conditions fixed meets those stipulations, since the principal symbol of \(D_E^-\), as a twisted Dirac operator, is given by Clifford multiplication on \(S^-(TN)\otimes E\) and varies continuously with the metric \(\g_N.\)

Thus, since \(B\) depends only on \(f|_{\partial N}\) and its tangential derivatives, along with \(\g_N|_{\partial N}\) and \(\g_M|_{\partial M}\), we may deform \(\g_N\) without changing the index and assume that \(\g_N|_U=\g_N|_{\partial N}+\dd r^2\) with respect to the identification \(U\cong \partial N\times [0,\varepsilon]\),  deform \(f|_{U}\) so that it is a diffeomorphism onto its image, and deform \(\g_M\) so that \(f|_{U}\) is an isometry.  It follows that \(\nabla^{S_N}\) and \(f^*\nabla^{S_M}\) are isomorphic and independent of \(r\). Since \(\g_N\)  is of product form near \(\partial N,\) the slices  \(\partial N\times\{r\}\) are totally geodesic and \(\nabla^{S_N}\) is equal to \(\nabla^{\partial N}\) when restricted to each, so \(\nabla=\nabla^\partial.\)   
Thus \(A=0\), and, near the boundary,
\[D_E^-=\nu(\partial_r+\nu B\nu).\]
This is the product form and boundary conditions used by Atiyah--Patodi--Singer for the operator \(D_E^-\) with boundary operator \(\nu B\nu\) and adjoint \(D_E^+\). (Note that the conditions for \(D_E^+\) and \(D_E^-\) differ by the role of the strict and nonstrict inequalities.) Thus, applying  the Atiyah--Patodi--Singer Index Theorem~\cite[Thm.~4.2]{APS} to \(D_E^-\), 
\begin{equation}\label{eq:APS-1}
\begin{aligned}
\text{ind}\big(D_E^+\big)=-\text{ind}\big(D_E^-\big)&=-\int_{N} a(D_E^-)+\frac{h(\nu B\nu)+\eta(\nu B\nu)}{2}\\
&=\int_{N} a(D_E^+)+\frac{h(B)-\eta(B)}{2},
\end{aligned}
\end{equation}
where \(a\) is the integrand in the Atiyah--Singer Index Theorem (in the closed case), which has opposite signs for adjoints, \(h(\cdot)=\dim\ker(\cdot)\) is the dimension of the kernel, and \(\eta\) is the \(\eta\)-spectral invariant, which commutes with sign changes. (Note that \(\nu B \nu=-(\nu^{-1}B\nu)\) has opposite spectrum to \(B\).)
    
By the discussion in the closed case (see \Cref{kernel2}), we have that
\[a(D_E^+)=-\tfrac{1}{12} p_1(\g_N) +\tfrac{1}{4} f^*p_1(\g_M)+\tfrac{1}{2} f^*e(\g_M),\]
where \(p_1\) and \(e\) denote the Pontryagin and Euler forms from Chern--Weil theory. Since \(f|_{\partial N}\) is an oriented isometry, there is an isomorphism of \(S^+(TN)|_{\partial N} \cong f^*S^+(TM)|_{\partial M}\) identifying \(\nabla^{\partial N}\) with \(f^*\nabla^{\partial M}\).
Under the isomorphisms
\[S^+(TN)\otimes E|_{\partial N}\cong S^+(TN)\otimes S^+(TN)|_{\partial N}\cong \wedge_\C^{+,\text{even}}TN^*|_{\partial N}\cong\wedge_\C^\text{even}T\partial N^*,\]
where the final isomorphism is given by restriction, and its inverse on \(\wedge_\C^{2p}T\partial N^*\) is given by \(\alpha\mapsto \alpha-({-1})^{p}\,* \alpha\), the operator \(B\) corresponds to \((-1)^p(*_{\partial N}\dd_{\partial N}-\dd_{\partial N}*_{\partial N})\) on \(\wedge_\C^{2p}T\partial N^*.\) The latter is the boundary operator in the Atiyah--Patodi--Singer Signature Theorem~\cite[Thm.~4.14]{APS}, from which we obtain that
\[\int_N\tfrac13 p_1(\g_N) =\sigma(N)+\eta(B), \qquad\text{and}\qquad \int_M\tfrac13 p_1(\g_M) =\sigma(M)+\eta(B),\]
where we are using that \(f\) is an isometry on the boundary, hence both boundary operators have the same spectrum. Moreover, since \(f\) is an orientation-preserving diffeomorphism near the boundary, it follows that
\begin{equation}\label{eq:intaDE}
\begin{aligned}
\int_N a(D_E^+)&=-\int_N\tfrac{1}{12} p_1(\g_N)+\int_M\tfrac14 p_1(\g_M) + \tfrac12 e(\g_M)\\
&=-\tfrac14 {\sigma(N)}+\tfrac34 \sigma(M)+\tfrac12 \eta(B)+\tfrac12 \chi(M).
\end{aligned}
\end{equation}
Thus, \eqref{eq:ind-bdy} follows from \eqref{eq:APS-1} and \eqref{eq:intaDE}, as the kernel of \(\pm(*_{\partial N}\dd_{\partial N}-\dd_{\partial N}*_{\partial N})\) is the space of even harmonic forms on $\partial N\cong\partial M$, and hence \(h(B)=b_0(\partial M)+b_2(\partial M)\) by Hodge theory.
\end{proof}

\subsection{Extremality and Rigidity}
We now combine the Bochner--Lichnerowicz--Weitzenb\"ock formula with \Cref{lem:section} to prove our main result on area-extremality and area-rigidity of compact $4$-manifolds with boundary (see \Cref{def:extremality-local}), closely following the proof of \Cref{thm}, but also handling boundary terms.

\begin{theorem}\label{local}
A compact oriented Riemannian $4$-manifold \((M,\g_M)\) 
with $\sec\geq0$ such that $R+\tau\,*\succeq0$ for a nonpositive $\tau\colon M\to \R$ and convex boundary $\partial M$, i.e., \(\II_{\partial M}\succeq0\), is area-extremal with respect to
\begin{equation*} 
\mathcal C_{\partial}=\left\{f\colon (N,\g_N)\to (M,\g_M)  :\!\! \begin{array}{l}
	f|_{\partial N}\text{ is an oriented isometry onto \(\partial M\), and}\\
	 2\chi(M)+3\sigma(M)+2b_0(\partial M)+2b_2(\partial M)>\sigma(N)
\end{array}\!\!\!\right\}.
\end{equation*}
If, in addition, $\frac{\scal(\g_M)}{2}\g_M\succ\Ric(\g_M)\succ0$ at $q\in M$, then $(M,\g_M)$ is area-rigid at $q\in M$ with respect to $\mathcal C_{\partial}$.
\end{theorem}

\begin{proof} 
   Let $f\colon (N,\g_N)\to (M,\g_M)$ be a competitor in $\mathcal C_{\partial}$, and let $E=f^*S^+(TM)$. As in \Cref{thm}, working pointwise with $W=T_pN$, $V=T_{f(p)}M$, $l\colon W\to V$ given by $\dd f(p)$, and $L=\wedge^2 l$, we have from \eqref{d2} that 
    \[ D_E^2=\nabla^*\nabla+\tfrac{1}{4}\scal(\g_N)+\mathcal{R}(R, L).\]
        Let \(\xi\in\Gamma(S^+(TN)\otimes E)\) be a nonzero section in the kernel of \(D_E\), which exists by \Cref{lem:section}. Using \Cref{rt,fscal,nneg,star} as in the proof of \Cref{thm}, we have:
    \begin{equation}\label{lich}
    \begin{aligned}
    0&=\int_N\langle\nabla^*\nabla\xi,\xi\rangle+\tfrac14\int_N \scal(\g_N)\|\xi\|^2 +\int_N\left<\mathcal{R}(R,L)\xi,\xi\right>\\
    &\geq \int_N\langle\nabla^*\nabla\xi,\xi\rangle+\tfrac14\int_N\big(\scal(\g_N)-\scal(\g_M)\circ f\big)\|\xi\|^2.
    \end{aligned}
    \end{equation}

    A standard computation using the divergence theorem implies that
    \begin{equation}\label{adjoint}
    \int_N\langle\nabla^*\nabla\xi,\xi\rangle=\int_N\|\nabla\xi\|^2+\int_{\partial N}\langle \nabla_\nu\xi,\xi\rangle,
    \end{equation}
    where, as above, $\nu$ is the inward unit normal along $\partial N$.
    In a local \(\g_N\)-orthonormal frame \(\{e_i\}\) of $N$ with \(e_1=-\nu\), since \(D_E\xi=0,\) we may use \eqref{eq:BtildeA} to compute
    \begin{equation}\label{eq:nablanuxi}
    \nabla_\nu\xi=\textstyle\sum\limits_{i=2}^4\nu e_i\nabla_{e_i}\xi=-B\xi- A\xi.
    \end{equation}
    As \(\xi\in \Gamma^+,\) we have \(\langle B\xi,\xi\rangle\leq 0\) from \eqref{eq:gamma+def}. To analyze the endomorphism \(A\) at the boundary, let \(\widetilde\nu\) be an the inward unit normal field of \(\partial M\). For all \(v\in T\partial N\),
    \begin{align}\label{sff}
    \nabla_{v}^{S_N}-\nabla_v^{\partial N}&=\tfrac12 \II_{\partial N}(v)\cdot \nu\in \cl(TN)|_{\partial N},\\
    \nabla_{l(v)}^{S_M}-\nabla_{l(v)}^{\partial M}&=\tfrac12 \II_{\partial M}(l(v))\cdot \widetilde\nu\in \cl(TM)|_{\partial M},
	\end{align}
    where we use the same symbol \(\II_{\partial N}\) to denote the symmetric endomorphism of \(T\partial N\) induced by $\II_{\partial N}$, namely so that \(\II_{\partial N}(x,y)=\g_N(\nabla^{N}_x y,\nu )=\g_N (\II_{\partial N}(x),y)\) for all \(x,y\in T\partial N,\) and similarly for \(\II_{\partial M}\). Using the above and \eqref{eq:BtildeA}, we compute
    \begin{align*}
    A=&-\nu\sum_{i=2}^{4}e_i\left(\nabla_{e_i}^{S_N}\otimes 1+1\otimes \nabla^{S_M}_{l(e_i)}-\nabla^{\partial N}_{e_i}\otimes 1-1\otimes \nabla^{\partial M}_{l(e_i)}\right)\\
    =&\tfrac12 \sum_{i=2}^{4} e_i\, \II_{\partial N}(e_i)\otimes 1+\tfrac12 \sum_{i=2}^{4}\nu\, e_i\otimes{\widetilde\nu}\, \II_{\partial M}(l(e_i)).
    \end{align*}
    
    By the symmetries of \(\II_{\partial N}\) and of Clifford multiplication, the first term above is \(-\tfrac12 H(\g_N).\) The second term can be written using the pointwise formalism from \Cref{pi}, as follows. Given $p\in \partial N$, since \(f|_{\partial N}\colon\partial N\to\partial M\) is an isometry and $\{e_2,e_3,e_4\}$ is an orthonormal frame of $T_p\partial N\subset W$, its image \(\{l(e_2),l(e_3),l(e_4)\}\) is an orthonormal frame of \(T_{f(p)}\partial M\subset V\). Let \(Q\colon \wedge^2 V\to \wedge^2 V\) be the symmetric endomorphism such that \(Q(\widetilde\nu\wedge l(e_i))=\widetilde\nu\wedge \II_{\partial M}(l(e_i))\) and \(Q(l(e_i)\wedge l(e_j))=0\), for \(2\leq i,j\leq 4,\) and note that $Q\in \Sym^2_b(\wedge^2 V)$, i.e., $Q$ satisfies the first Bianchi identity. Moreover, let \(l'\colon W \to V\) be the linear map such that \(l'(e_i)=l(e_i)\) for \(2\leq i\leq 4\) and \(l'(\nu)=\widetilde \nu\), and set $L'=\wedge^2 l'$. 
    From \Cref{def:RT}, we have that
    \[-\mathcal{R}(Q,L')=\sum_{i<j}e_ie_j\otimes Q(l(e_i)\wedge l(e_j))=\tfrac12\sum_{j\geq2}\nu \, e_j\otimes \widetilde\nu \,\II_{\partial M}(l(e_j)),\]
    so we conclude that \(A=-\tfrac12 H(\g_N)-\mathcal{R}(Q,L').\) Applying \Cref{rt,fscal,nneg}, since \(\II_{\partial M}\succeq0\) and \(l'\) is an isometry, it follows that \(Q\succeq0\) and
    \[-\langle A\xi,\xi\rangle\geq\tfrac12 \big(H(\g_N) - \operatorname{tr}(Q) \big)\|\xi\|^2=\tfrac12\big( H(\g_N)-H(\g_M) \big)\|\xi\|^2.\]
    From \eqref{adjoint} and the above, the left-hand side of \eqref{eq:nablanuxi} can be bounded from below:
    \begin{align*}
    \int_N\langle \nabla^*\nabla\xi,\xi\rangle&=\int_N  \|\nabla\xi\|^2-\int_{\partial N}\langle B\xi,\xi\rangle-\int_{\partial N}\langle A\xi,\xi\rangle\\
    &\geq\int_N \|\nabla\xi\|^2+\tfrac12\int_{\partial N} \big(H({\g_N})-H({\g_M})\circ f\big)\|\xi\|^2.
    \end{align*}    
    Combined with \eqref{lich}, this yields:
    \begin{align*}
    0&\geq\int_N\|\nabla\xi\|^2+\tfrac14\int_N\big(\scal(\g_N)-\scal(\g_M)\circ f\big)\|\xi\|^2\\
    &\quad+\tfrac12\int_{\partial N}\big(H({\g_N})-H({\g_M})\circ f\big)\|\xi\|^2.
    \end{align*}    
    The stated conclusions now follow exactly as in the proof of \Cref{thm}.  
\end{proof} 

\begin{remark}\label{lott-comparison}
	The boundary conditions used to compute the index in \Cref{lem:section} are Atiyah--Patodi--Singer-type conditions defined in terms of the intrinsic Dirac operator on the boundary. This introduces terms involving the second fundamental form $\II$ in the above computation, hence into the hypotheses of \Cref{local}; similarly to \cite[Thm.~1.1]{lott-spin}. That is in contrast to \cite[Thm~1.3]{lott-spin}, where  Atiyah--Patodi--Singer conditions are used, defined in terms of the ambient Dirac operator restricted to the boundary. That method avoids the assumption $\II\succeq0$, but requires stronger assumptions on the map \(f\) at the boundary.

	Our method is necessary in order to work with the bundle \(S^+(TN)\otimes f^*S^+(TM)\), as required by our curvature condition. Even in the realm of metrics with $R\succeq0$, our result extends \cite{lott-spin} in dimension 4, as it yields a version of area-extremality with the geometric conditions of \cite{lott-spin} but with a topological assumption weaker than vanishing Euler characteristic, see \Cref{ex:examples-local}. Future work will attempt a similar generalization in higher dimensions.
\end{remark}

The following result implies \Cref{mainthm-local2} in the Introduction, since \(\mathcal C^\mathrm{id}_\partial\subset  \mathcal C_\partial^{\mathrm{self}}.\)

\begin{corollary}\label{cor:boundary-self}
   A compact oriented Riemannian $4$-manifold \((M,\g_0)\) with $\sec\geq0$ such that $R+\tau\,*\succeq0$ for a nonpositive or nonnegative $\tau\colon M\to \R$ 
 and convex boundary $\partial M$, i.e., \(\II_{\partial M}\succeq0\), is area-extremal with respect to
 \begin{equation}\label{eq:classC-local-self}
   \mathcal C_\partial^{\mathrm{self}}=\big\{f\colon (M,\g_1)\to (M,\g_0) \; :  \; f|_{\partial M}\text{ is an oriented isometry onto }\partial M \big\}.
   \end{equation}
\end{corollary}

\begin{proof}
    Choose the orientation of \(M\) for which \(\tau\leq0\), and let \(N\) be either \(M\) or \(\overline M\), i.e., $M$ with an orientation that will be fixed later. Let \(f\colon N\to M\) be a boundary-preserving spin map and \(\g_1\) be a metric on \(N\) such that 
    \[\scal(\g_1)\geq \scal(\g_0)\circ f, \quad \wedge^2\g_1\succeq f^*\wedge^2\g_0, \quad  H({\g_1})\geq H({\g_0})\circ f,\] and \(f|_{\partial N}\) is an isometry between \(\g_1|_{\partial N}\) and \(\g_0|_{\partial N}.\) Orient \(N\) such that  \(f|_{\partial N}\) \emph{preserves} orientation, i.e., is an oriented isometry. The conclusion will follow from \Cref{local} by showing that the topological condition in $\mathcal C_{\partial}$ is satisfied by $f\colon N\to M$. 
    
    By the Soul Theorem, see e.g.~\cite[Sec.~9-10]{eschenburg},
    since $(M,\g_0)$ has \(\sec\geq0\) and \(\II_{\partial M}\succeq0 \), it is diffeomorphic to the total space of a disk bundle over a closed totally geodesic submanifold \(\Sigma\subset M\). We proceed case-by-case in terms of \(0\leq \dim\Sigma\leq 3\).
    
    If \(\dim\Sigma\leq 1,\) then \(\chi(M)=\chi(\Sigma)\geq0\), and \(H^2(M)=H^2(\Sigma)=0\), which implies \(\sigma(M)=\sigma(N)=0.\)  Thus, the topological condition in $\mathcal C_{\partial}$ is satisfied, as \(b_0(\partial M)>0.\)
    
    If \(\dim\Sigma=3,\) then \(M\to\Sigma\) is an interval bundle which restricts to a covering map \(\partial M\to \Sigma.\) Such a map is homotopic to the inclusion \(\partial M\to M\), and thus \(H^2(M)\to H^2(\partial M)\) is injective and \(H^2(M,\partial M)\to H^2(M)\) is trivial. Therefore, \(\sigma(M)=\sigma(N)=0\), so the conclusion follows as in the previous case, since $\chi(\Sigma)=0$.
   
Finally, assume \(\dim\Sigma=2,\) in which case the Gauss--Bonnet Theorem implies that \(\chi(\Sigma)\geq0\). If \(\Sigma\) is nonorientable, then \(H^2(M)=H^2(\Sigma)=0\) and \(\sigma(M)=0\), so the topological condition in $\mathcal C_{\partial}$ is satisfied as in the previous cases. Else, \(\Sigma\) is orientable, so \(b_2(M)=b_2(\Sigma)=1\) and \(|\sigma(M)|\leq 1.\) If \(\Sigma\) is diffeomorphic to \(\S^2\), then \(2\chi(M)+3\sigma(M)-\sigma(N)=4+(3\pm1)\sigma(M)\geq0\) and hence the topological condition in $\mathcal C_{\partial}$ is again satisfied. If \(\Sigma\) is diffeomorphic to \(T^2\), then \(\partial M\) is diffeomorphic to an oriented \(\S^1\)-bundle over \(T^2\) so \(b_2(\partial M)\geq 2.\) Once again, the topological condition in $\mathcal C_{\partial}$ is satisfied, as \(2b_0(\partial M)+2b_2(\partial M)\geq 6\geq(-3\pm1)\sigma(M)\).
\end{proof}

Our final main result implies \Cref{mainthm-local1} in the Introduction because \(\mathcal C^\mathrm{id}_\partial\subset \mathcal C_\mathrm{loc}\) if $|\sigma(M)|<4$, and is captured by the vague but clear statement that, near a point where $\sec>0$, a Riemannian metric on a $4$-manifold cannot be modified to have greater scalar curvature without decreasing the area of some tangent $2$-plane. 

\begin{corollary}
    If \((X,\g)\) is a Riemannian \(4\)-manifold and \(p\in X\) is a point at which $\sec>0$, then there is a neighborhood \(M\cong D^4\) of \(p\) such that \(\g|_M\) is area-extremal with respect to
    \begin{equation*}
    \mathcal C_\mathrm{loc}=\left\{f\colon (N,\g_N)\to (M,\g|_M) :
    \!\! \begin{array}{l}
	f|_{\partial N}\text{ is an oriented isometry onto \(\partial M\),} \\
	 \text{and }|\sigma(N)|<4 \end{array} \right\}.
    \end{equation*}
\end{corollary}

\begin{proof}
    At each stage of the proof, we shrink \(M\) if needed, in order to make a series of assumptions. First, we may choose a neighborhood \(M\) of $p\in X$ such that $\sec>0$ at all points of $M$.  Then, by the Finsler--Thorpe trick (see \Cref{prop:FTtrick}), there exists a function \(\tau\colon M\to \R\) such that the curvature operator \(R\) of \((M,\g|_M)\) satisfies \(R+\tau\,* \succ0.\)
    Moreover, there is an open interval $(\tau_\mathrm{min}(x),\tau_\mathrm{max}(x))$ of possible values for $\tau$ at each point $x\in M$ that depends continuously on $x$, so we may choose $\tau$ to be continuous and \(\tau(p)\neq0.\) Shrinking \(M\) if necessary, we may assume \(\tau\neq 0\) throughout \(M,\) and, changing the orientation of \(M\) if necessary, we may assume \(\tau<0\) on all $M$. Finally, we shrink \(M\) to be the closure of a convex ball in $X$ containing $p\in X$. Then \((M,\g|_M)\) satisfies all the hypotheses of \Cref{local} and, since \(3\sigma(M)+2\chi(M)+2b_0(\partial M)+2b_2(\partial M)=4,\) the corollary follows.
\end{proof}

\begin{example}\label{ex:examples-local}
Let \(\S^n_+\) denote the unit \(n\)-dimensional hemisphere. Standard metrics on \(\S^4_+,\) \(\S^1\times\S^3_+,\) \(\S^2\times \S^2_+\), and \(\S^3\times [-1,1]\) have positive-semidefinite curvature operator and totally geodesic boundary, and are hence area-extremal with respect to $\mathcal C_\partial$ by \Cref{local}. In particular, they are area-extremal with respect to \(\mathcal{C}_\partial^\mathrm{self}\subset\mathcal C_\partial\) defined in \eqref{eq:classC-local-self}. The first and third spaces are also area-rigid, as they satisfy the additional hypothesis $\frac{\scal}{2}\g\succ\Ric\succ0$.  Note that the second and fourth spaces have vanishing Euler characteristic. Similar results, with different boundary conditions, are implied by \cite[Thm.~1.3]{lott-spin}; see also \Cref{lott-comparison}.

The $\mathsf{SU}(2)$-invariant metrics with totally geodesic boundary on the normal disk bundle \(\nu(\C P^1)\) of $\C P^1\subset\C P^2$, described in the proof of \Cref{thm:cheeger-metrics}, also satisfy the hypotheses of \Cref{local}. Thus, these metrics are area-extremal with respect to \(\mathcal{C}_\partial\) and \(\mathcal{C}_\partial^\mathrm{self}\). Moreover, if such a metric has vanishing neck length, then it is area-rigid with respect to diffeomorphisms. Indeed, the neck in that case consists solely of the boundary $\partial \nu(\C P^1)$ and $\frac{\scal}{2}\g\succ\Ric\succ0$ everywhere else, so the same continuity argument in \Cref{rem:noneck} applies. This example is notable because \(\nu(\C P^1)\) does not admit metrics with both \(R\succeq 0\) and \(\II_{\partial\nu(\C P^1)}\succeq0,\) to which previous methods could be applied to prove area-extremality: a simply-connected manifold admitting such a metric is diffeomorphic to a \emph{trivial} disk bundle over its soul, see e.g.~\cite{Noronha}.
\end{example}

\providecommand{\bysame}{\leavevmode\hbox to3em{\hrulefill}\thinspace}
\providecommand{\MR}{\relax\ifhmode\unskip\space\fi MR }
\providecommand{\MRhref}[2]{%
  \href{http://www.ams.org/mathscinet-getitem?mr=#1}{#2}
}
\providecommand{\href}[2]{#2}

\end{document}